\documentclass{amsart}
\usepackage[a4paper, margin=2cm]{geometry}
\usepackage{graphicx}
\usepackage{amssymb,amsmath}
\usepackage{amsthm,thmtools}
\usepackage{mathrsfs}
\usepackage{array}

\usepackage{tikz}
\usetikzlibrary{arrows,arrows.meta,decorations.pathreplacing,decorations.markings,shapes,calc}
\tikzset{labelsize/.style={font=\scriptsize}}
\tikzset{string/.style={very thick}}
\tikzset{
  pto/.style={->,postaction={decorate},
    decoration={
        markings,
        mark=at position 0.5 with {\arrow{|}}}
  },
}
\tikzset{2cell/.style={-implies,double,double equal sign distance,shorten >=9pt, shorten <=10pt}}

\newcolumntype{N}{@{}m{0pt}@{}}

\declaretheorem[style=plain,numberwithin=section,name=Theorem]{theorem}

\declaretheorem[style=plain,sibling=theorem,name=Proposition]{proposition}

\declaretheorem[style=definition,qed=$\blacksquare$,sibling=theorem,name=Definition]{definition}
\declaretheorem[style=definition,qed=$\blacksquare$,sibling=theorem,name=Example]{example}
\declaretheorem[style=definition,qed=$\blacksquare$,sibling=theorem,name=Remark]{remark}

\usepackage[colorlinks=true]{hyperref}

\mathchardef\mhyphen="2D

\newcommand{\Set}{\mathbf{Set}}

\newcommand{\Catone}{\mathbf{Cat}}
\newcommand{\CAT}{{\mathscr{C\!A\!T}}}

\newcommand{\colim}{{\mathrm{colim}}}

\newcommand{\cat}[1]{{\mathcal{#1}}}

\newcommand{\N}{\mathbb{N}}

\newcommand{\nameof}[1]{{\lceil #1\rceil}}

\newcommand{\Yoneda}{\mathbf{y}}

\newcommand{\Contr}{{\mathbf{Contr}}}

\newcommand{\OC}[2]{{\mathbf{OC}(#2)}}

\newcommand{\G}{\mathbb{G}}
\newcommand{\dgraph}[2]{
\begin{tikzpicture}[baseline=-\the\dimexpr\fontdimen22\textfont2\relax ]
	  \node (o) {};
      \node (v) [left =  0.15cm of o]{$#1$};
      \node (e) [right = 0.15cm of o] {$#2$};
      
      \draw [transform canvas={yshift=0.3em},<-]  (v) to (e);
      \draw [transform canvas={yshift=-0.3em},<-] (v) to (e);
\end{tikzpicture}}

\newcommand{\enGph}[1]{{{#1}\mhyphen\mathbf{Gph}}}
\newcommand{\enCat}[1]{{{#1}\mhyphen\mathbf{Cat}}}

\newcommand{\op}{\mathrm{op}}

\newcommand{\phileft}{{\enGph{\omega}/\phi^{(\omega)}_1}}

\newcommand{\Us}{{U_\mathrm{s}}}

\newcommand{\ev}{\mathrm{ev}}
\newcommand{\cart}{{\mathrm{cart}}}

\newcommand{\ob}[1]{{\mathrm{ob}(#1)}}
\newcommand{\ar}[1]{{\mathrm{ar}_{#1}}}

\newcommand{\WkCats}[1]{{\mathbf{Wk}\mhyphen{#1}\mhyphen\mathbf{Cat}_\mathrm{s}}} 
\newcommand{\WkCat}[1]{{\mathbf{Wk}\mhyphen{#1}\mhyphen\mathbf{Cat}}} 

\newcommand{\WkGpds}[1]{{\mathbf{Wk}\mhyphen{#1}\mhyphen\mathbf{Gpd}_\mathrm{s}}} 

\newcommand{\WkGpd}[1]{{\mathbf{Wk}\mhyphen{#1}\mhyphen\mathbf{Gpd}}} 

\newcommand{\Alg}[1]{{{#1}\mhyphen\mathbf{Alg}}}

\newcommand{\Mon}[1]{{\mathbf{Mon}(#1)}}

\newcommand{\id}{\mathrm{id}}
\newcommand{\Id}{\mathrm{Id}}

\makeatletter
\DeclareRobustCommand{\rvdots}{%
  \vbox{
    \baselineskip4\p@\lineskiplimit\z@
    \kern-\p@
    \hbox{.}\hbox{.}\hbox{.}
  }}
\makeatother

\makeatletter
\newcommand{\pto}{}
\newcommand{\pgets}{}
\DeclareRobustCommand{\pto}{\mathrel{\mathpalette\p@to@gets\to}}
\DeclareRobustCommand{\pgets}{\mathrel{\mathpalette\p@to@gets\gets}}
\newcommand{\p@to@gets}[2]{%
  \ooalign{\hidewidth$\m@th#1\mapstochar\mkern5mu$\hidewidth\cr$\m@th#1\longrightarrow$\cr}%
}
\makeatother

\newcommand{\defemph}[1]{\textbf{#1}}

\begin{document}
\title{Hom weak $\omega$-categories of a weak $\omega$-category}
\author{Thomas Cottrell}
  \address{Department of Mathematical Sciences, University of Bath, Bath, United Kingdom}
  \email{t.p.cottrell@bath.ac.uk}

\author{Soichiro Fujii}
\address{Research Institute for Mathematical Sciences, Kyoto University, Kyoto, Japan}
\email{s.fujii.math@gmail.com}
\date{\today}
\thanks{We gratefully acknowledge the support of Royal Society grant IE160402. The second author is supported by ERATO HASUO Metamathematics for Systems Design Project (No.~JPMJER1603), JST. No data were generated in association with this paper.}

\keywords{Weak $\omega$-category, weak $\omega$-groupoid, weak $\omega$-functor, operad, intensional Martin-L\"of type theory, identity type.}
\subjclass[2020]{18N65, 18N20}
\dedicatory{Dedicated to John Power on the occasion of his 60th birthday.}

\begin{abstract}
Classical definitions of weak higher-dimensional categories are given inductively; for example, a bicategory has a set of objects and hom categories, and a tricategory has a set of objects and hom bicategories. However, more recent definitions of weak $n$-categories for all natural numbers $n$, or of weak $\omega$-categories, take more sophisticated approaches, and the nature of the ``hom’’ is often not immediate from the definitions. In this paper, we focus on Leinster’s definition of weak $\omega$-category based on an earlier definition by Batanin, and construct for each weak $\omega$-category $\mathcal{A}$, an underlying (weak $\omega$-category)-enriched graph consisting of the same objects and for each pair of objects $x$ and $y$, a hom weak $\omega$-category $\mathcal{A}(x,y)$. We also show that our construction is functorial with respect to weak $\omega$-functors introduced by Garner.
\end{abstract}

\maketitle

\section{Introduction}
\label{sec:intro}
\emph{Identity types} are one of the most interesting features of (intensional) Martin-L\"of type theory \cite{NPS_programming}. 
Given any type $A$ and pair of terms $a,b$ of that type, this feature yields the type $\Id_A(a,b)$, which, under the Curry--Howard correspondence, may be regarded as the {logical} proposition $a=b$. Since identity types are themselves types, we can also obtain types of the form $\Id_{\Id_A(a,b)}(p,q)$.
This process leads to an infinite hierarchy of iterated identity types, which has certain structure reminiscent of the familiar reflexivity, symmetry and transitivity properties of equality. 
In fact it forms an internal \emph{weak $\omega$-groupoid} \cite{vandenBerg_Garner,Lumsdaine},
thus suggesting connections to higher-dimensional category theory and homotopy theory \cite{hottbook}.

In order to prove {that certain principles (such as Uniqueness of Identity Proofs) are undecidable in} Martin-L\"of type theory, Hofmann and Streicher \cite{Hofmann_Streicher_groupoid} gave a groupoid model of it. 
Then Warren \cite{Warren_strict_omega_groupoid} gave a strict $\omega$-groupoid model, {which refutes further principles (such as the truncation rules)}. 
In both cases, the interpretations of identity types were given by means of homs; in the former case hom sets of a groupoid seen as discrete groupoids, and in the latter case hom strict $\omega$-groupoids of a strict $\omega$-groupoid. 
With the abovementioned observation that types have the structure of internal weak $\omega$-groupoids in mind, it seems natural to seek a model of Martin-L\"of type theory using (external) weak $\omega$-groupoids, or at least a suitable subclass of them.
This paper is a contribution to this goal; here we shall establish a crucial step for it, by showing that an {(external)} weak $\omega$-groupoid indeed has {(external)} weak $\omega$-groupoids as homs.

\medskip

Weak $\omega$-groupoids are weak $\omega$-categories in which each $k$-cell ($k\geq 1$) is weakly invertible, and in turn, weak $\omega$-categories are a higher-dimensional analogue of categories in which one has $k$-cells for each $k\in\N$ and various composition operations, satisfying the usual category axioms up to coherent higher-dimensional cells. Weak $\omega$-categories can be thought of as a limit of weak $n$-categories for $n\in\N$, in which one only has $k$-cells for $k\leq n$.
Classical definitions of weak $n$-category for small $n$ are well-known: the cases $n=0,1,2$ and $3$ correspond to set, category, bicategory \cite{Benabou_bicat} and tricategory \cite{GPS} respectively.

Subsequently, several definitions of weak $n$-category for arbitrary $n\in\N$, as well as of weak $\omega$-category, have been proposed by various authors; see \cite{Leinster_survey}. In this paper we shall focus on the definition given by Leinster~\cite{Leinster_book} following an earlier definition by Batanin \cite{Batanin_98}, this being the one adopted in the abovementioned papers \cite{vandenBerg_Garner,Lumsdaine}. Hence, more specifically, our main aim is to show that each weak $\omega$-category in the sense of Leinster has weak $\omega$-categories (again, in the sense of Leinster) as homs; a completely parallel argument also shows that a weak $(n+1)$-category has weak $n$-categories as homs for each $n\in\N$. We then observe that this result suitably restricts to weak $\omega$-groupoids.

The reader familiar with the classical definitions of bicategory or tricategory might guess that such a result would be immediate from the definition, as, for example, a tricategory $\cat{T}$ is \emph{defined} as the data consisting of a set of objects $\ob{\cat{T}}$, and for each pair $x,y$ of objects, a \emph{hom bicategory} $\cat{T}(x,y)$, together with various composition operations and coherence cells.  However, the weakened enrichment approach used in the classical definitions requires the coherence cells to be specified explicitly.  
The number of coherence cells needed, and the complexity of those cells, increases rapidly as the dimension increases, and it therefore becomes impractical to maintain this approach when defining weak $n$-categories for arbitrary $n$ or $n=\omega$.

Consequently, the various proposed definitions of weak $n$-category---including Leinster's definition---do not intrinsically include hom weak $(n - 1)$-categories, and they are so different in style to the classical definitions that our result is not at all immediate.\footnote{One arguable exception is Trimble's definition, which does take a weakened enrichment approach.  However, Trimble's approach differs from the classical definitions in other important ways.  In particular the resulting $n$-categories are not fully weak as they do not have weak interchange, and indeed Trimble described them as ``flabby $n$-categories'' rather than ``weak $n$-categories'' \cite{Leinster_survey}.}  Indeed, to the best of our knowledge no one seems to have written down a detailed proof of this result, either in the case of $\omega$ or $n\in\N$.  Our result would be vital not only for a semantics of identity types, but also for any serious development of a theory of weak higher-dimensional categories.

Leinster's definition of weak $\omega$-category starts from the underlying structure of \emph{$\omega$-graph}, which simply consists of cells of various dimensions together with suitable boundary (source and target) information. The weak $\omega$-categories are defined as the Eilenberg--Moore algebras of a suitable monad on the category of $\omega$-graphs; using the terminology we shall explain in Section \ref{sec:Leinsters_def}, this monad is induced by the \emph{initial $T^{(\omega)}$-operad with contraction}.
We denote the Eilenberg--Moore category of the monad for weak $\omega$-category by $\WkCats{\omega}$. 
Our main construction amounts to defining a suitable ``forgetful'' functor 
\[
\Us\colon\WkCats{\omega}\longrightarrow \enGph{(\WkCats{\omega})},
\]
where the codomain is the category of ($\WkCats{\omega}$)-enriched graphs (cf.~Definition \ref{def:V-graph}); applying $\Us$ to a weak $\omega$-category $\cat{A}$, we obtain the ($\WkCats{\omega}$)-enriched graph consisting of the same objects and for each pair of objects $x$ and $y$, a hom weak $\omega$-category $\cat{A}(x,y)$.
The key to the definition of $\Us$ is the observation that both the domain and the codomain of $\Us$ are monadic over the category of $\omega$-graphs. We shall induce $\Us$ from a canonical monad morphism, which is ultimately induced by the \emph{initiality} used to determine the monad for weak $\omega$-category.

The morphisms of the Eilenberg--Moore category $\WkCats{\omega}$ are called \emph{strict $\omega$-functors}, because they preserve the structures of weak $\omega$-categories strictly. 
In the context of weak higher-dimensional categories, however, a more natural notion of functor is one that preserves the structures up to coherent weakly invertible cells.
Garner \cite{Garner_homomorphisms} introduced such functors between weak $\omega$-categories, which we call \emph{weak $\omega$-functors}.
Our construction of hom weak $\omega$-categories is compatible with weak $\omega$-functors. That is, denoting the category of weak $\omega$-categories and weak $\omega$-functors by $\WkCat{\omega}$, we show that the functor $\Us$ extends to a functor
\[
U\colon \WkCat{\omega}\longrightarrow\enGph{(\WkCat{\omega})}.
\]
This means that from a weak $\omega$-functor $F\colon \cat{A}\longrightarrow\cat{B}$, we can extract its action on homs as a family of weak $\omega$-functors $(F_{x,y}\colon\cat{A}(x,y)\longrightarrow\cat{B}(Fx,Fy))_{x,y\in\ob{\cat{A}}}$.

\medskip

In Section \ref{sec:Leinsters_def} we review Leinster's definition of weak $\omega$-category.
Then in Section \ref{sec:forgetful_strict} we  construct the forgetful functor $\Us$. 
Section \ref{sec:Garner_weak_functor} provides a definition of weak $\omega$-functor due to Garner. In Section \ref{sec:forgetful_weak} we construct the extension $U$ of $\Us$. 
In the final Section \ref{sec:wk_omega_groupoid} (which can be read independently of Sections \ref{sec:Garner_weak_functor} and \ref{sec:forgetful_weak}) we review a definition of weak $\omega$-groupoid and observe that our construction suitably restricts to this case.

{
\subsection*{Related work}
At the beginning of the introduction, we mentioned several papers on the border between higher-dimensional category theory and type theory, in order to motivate the problem we shall treat in this paper from a computer science perspective. Although the rest of this paper does not use any type theory, we shall mention another, related line of research connecting these two subjects. The papers we have in mind aim to formalise the definitions of weak $\omega$-groupoid \cite{Altenkirch_Rypecek,Brunerie_thesis} and of weak $\omega$-category \cite{Finster_Mimram,Benjamin_thesis,Benjamin_Finster_Mimram_weak_omcat} via suitable dependent type theories.
(As Cartmell \cite{Cartmell_GAT} has shown, dependent type theories can be used as presentations of (generalised) algebraic theories.)

In a sense, the papers in this line of research are complementary to those mentioned at the beginning of the introduction: the former capture given higher-dimensional categorical structures by designing suitable type theories, whereas the latter relate suitable higher-dimensional categorical structures to a given type theory, namely Martin-L\"of type theory. The definitions of weak $\omega$-groupoid and of weak $\omega$-category adopted in the papers \cite{Brunerie_thesis,Finster_Mimram,Benjamin_thesis,Benjamin_Finster_Mimram_weak_omcat} are those of Grothendieck--Maltsiniotis \cite{Grothendieck,Maltsiniotis}, whose relationship to the Batanin--Leinster definitions has been studied by Ara \cite{Ara_thesis}. 
}

\subsection*{Acknowledgement}
We thank John Power for the countless discussions on weak higher-dimensional categories we have had in Bath, Tokyo and Kyoto. 
{Anonymous referees provided us with detailed lists of comments, which helped to improve the presentation.}

\section{Leinster's definition of weak \texorpdfstring{$\omega$}{omega}-category}
\label{sec:Leinsters_def}
In this section we review Leinster's definition
of weak $\omega$-category \cite{Leinster_book}, which was introduced as a variant of an earlier definition by Batanin \cite{Batanin_98}.
According to this definition, the weak $\omega$-categories are the Eilenberg--Moore algebras of a monad on the category $\enGph{\omega}$ of \emph{$\omega$-graphs}, which is (up to equivalence) the presheaf category over a simple category.
The monad for weak $\omega$-categories is defined by means of the following two notions:
(i)~\emph{globular operads},  and 
(ii)~\emph{contractions}.
For conceptual clarity, we present them as instances of simpler notions, namely:
(i')~\emph{operads over a cartesian monad},
following \cite{Leinster_book},
and (ii')~\emph{choices of diagonal-fillers}, following 
\cite{Garner_univ}.

\subsection{Operads over a cartesian monad}
\label{subsec:operad_over_cart_monad}
For any category $\cat{C}$, the category $[\cat{C},\cat{C}]$ of all endofunctors on it admits a natural (strict) monoidal structure, given by composition of endofunctors. Monoids in $[\cat{C},\cat{C}]$ are monads on $\cat{C}$. 
Now let $\cat{C}$ be a category with finite limits. 
We consider the monoidal subcategory $[\cat{C},\cat{C}]_\cart$ of $[\cat{C},\cat{C}]$ defined as follows.
\begin{itemize}
	\item An endofunctor on $\cat{C}$ is in $[\cat{C},\cat{C}]_\cart$ if and only if it preserves all pullbacks. 
	\item A natural transformation between pullback preserving endofunctors on $\cat{C}$ is in $[\cat{C},\cat{C}]_\cart$ if and only if it is \defemph{cartesian}, meaning that all its naturality squares are pullback squares in $\cat{C}$.
\end{itemize}
Monoids in $[\cat{C},\cat{C}]_\cart$ are called \defemph{cartesian monads} on $\cat{C}$.

Let us recall some standard facts about slice categories of a monoidal category.
\begin{proposition}
\label{prop:slice_of_mon_cat}
	Let $\cat{M}$ be a monoidal category and $T$ be a monoid in $\cat{M}$.
	\begin{enumerate}
		\item The slice category $\cat{M}/T$ acquires a canonical monoidal structure in such a way that the forgetful functor $\cat{M}/T\longrightarrow \cat{M}$ is strict monoidal.
		\item {The canonical functor $\Mon{\cat{M}/T}\longrightarrow \Mon{\cat{M}}/T$ is an isomorphism of categories. That is, to give a monoid in $\cat{M}/T$ is equivalent to give a monoid $T'$ in $\cat{M}$ together with a monoid morphism $T'\longrightarrow T$.}
	\end{enumerate}
\end{proposition}

Suppose $T=(T,\eta,\mu)$ is a cartesian monad on $\cat{C}$.
By Proposition \ref{prop:slice_of_mon_cat} (i), we obtain a monoidal category $[\cat{C},\cat{C}]_\cart/T$, and by Proposition \ref{prop:slice_of_mon_cat} (ii), a monoid in $[\cat{C},\cat{C}]_\cart/T$ is equivalent to a cartesian monad $T'$ on $\cat{C}$ equipped with a \emph{cartesian} monad morphism $T'\longrightarrow T$ (i.e., a monad morphism which is cartesian as a natural transformation).

Let $1$ denote the terminal object of $\cat{C}$.
The fundamental fact for the theory of operads over a cartesian monad is that the functor 
\begin{equation}
\label{eqn:ev_1}
	\ev_1\colon [\cat{C},\cat{C}]_\cart/T\longrightarrow \cat{C}/T1
\end{equation}
mapping $(\phi\colon F\longrightarrow T)\in [\cat{C},\cat{C}]_\cart/T$ to $(\phi_1\colon F1\longrightarrow T1)\in \cat{C}/T1$, is an equivalence of categories \cite{Kelly_club_data_type}.
The quasi-inverse of $\ev_1$ is given by mapping $(\ar{P}\colon P\longrightarrow T1)\in\cat{C}/T1$ to $(\overline{\ar{P}}\colon \overline{P}\longrightarrow T)\in [\cat{C},\cat{C}]_\cart/T$ determined by the pullback for each $C\in \cat{C}$:
\begin{equation*}
\begin{tikzpicture}[baseline=-\the\dimexpr\fontdimen22\textfont2\relax ]
      \node(00) at (-0.5,1) {$\overline{P}C$};
      \node(01) at (1.5,1) {$TC$};
      \node(10) at (-0.5,-1) {$P$};
      \node(11) at (1.5,-1) {$T1$};
      
      \draw [->] (00) to node[auto, labelsize] {$(\overline{\ar{P}})_C$} (01); 
      \draw [->] (01) to node[auto, labelsize] {$T!$} (11); 
      \draw [->] (00) to node[auto,swap,labelsize] {$\overline{P}!$} (10); 
      \draw [->] (10) to node[auto,swap,labelsize] {$\ar{P}$} (11); 
      \draw (-0.2,0.2) to (0.3,0.2) to (0.3,0.7);
\end{tikzpicture}
\end{equation*}
(the notation $\ar{P}$ is for \emph{arity}, because as we shall see later, this morphism may be interpreted as assigning arities to operations).
We shall frequently use this equivalence.

Transporting the {(strict)} monoidal structure on $[\cat{C},\cat{C}]_\cart/T$ along the equivalence \eqref{eqn:ev_1}, we obtain the following monoidal structure on $\cat{C}/T1$.
\begin{itemize}
	\item The unit is $I=(\eta_1\colon 1\longrightarrow T1)$.
	\item For any $(\ar{P}\colon P\longrightarrow T1)$
and $(\ar{Q}\colon Q\longrightarrow T1)$, 
their monoidal product is given by the top horizontal composite in the diagram below, i.e.,
$(P,\ar{P})\otimes(Q,\ar{Q})=(\mu_1\circ T(\ar{Q})\circ \partial_2\colon 
(P,\ar{P})\ast Q\longrightarrow T1)$:
\begin{equation}
\label{eqn:tensor_collection}
\begin{tikzpicture}[baseline=-\the\dimexpr\fontdimen22\textfont2\relax ]
      \node(00) at (-0.5,1) {$(P,\ar{P})\ast Q$};
      \node(01) at (2,1) {$TQ$};
      \node(10) at (-0.5,-1) {$P$};
      \node(11) at (2,-1) {$T1$.};
      \node(02) at (4,1) {$T^21$};
      \node(03) at (6,1) {$T1$};
      
      \draw [->] (00) to node[auto, labelsize] {$\partial_2$} (01); 
      \draw [->] (01) to node[auto, labelsize] {$T!$} (11); 
      \draw [->] (00) to node[auto,swap,labelsize] {$\partial_1$} (10); 
      \draw [->] (10) to node[auto,swap,labelsize] {$\ar{P}$} (11); 
      \draw [->] (01) to node[auto,labelsize] {$T(\ar{Q})$} (02); 
      \draw [->] (02) to node[auto,labelsize] {$\mu_1$} (03); 
      \draw (-0.2,0.2) to (0.3,0.2) to (0.3,0.7);
\end{tikzpicture}
\end{equation}
\end{itemize}
A monoid in $(\cat{C}/T1,I,\otimes)$ is called a \defemph{$T$-operad}.
{Notice that $(\cat{C}/T1,I,\otimes)$ is not a strict monoidal category (in general), since the equivalence \eqref{eqn:ev_1} is not an isomorphism (in general).}

\begin{example}
\label{ex:non-sym-operad}
The free monoid monad $(-)^\ast$ on $\Set$ is cartesian.
A $(-)^\ast$-operad is equivalent to a \emph{non-symmetric operad} \cite{May_loop}.
In this case, $T1\cong\N$, so an object of $\cat{C}/T1$ consists of a set $P$ whose elements are assigned natural number arities.  Thus we can view an element of $P$ of arity $k$ as an operation with $k$ inputs.
An element of $(P,\ar{P})\ast Q$ as in \eqref{eqn:tensor_collection} consists of an operation $p$ in $P$ together with a list of operations in $Q$ of length $\ar{P}(p)$.  The composite along the top of the diagram in \eqref{eqn:tensor_collection} adds together the arities of these operations in $Q$.
For each $(-)^\ast$-operad $((O,\ar{O}),e,m)$, the set $O$ can be understood as the set of all (derived) operations of the algebraic theory expressed by this operad, and $\ar{O}\colon O\longrightarrow 1^\ast\cong\N$ as mapping each operation to its arity.
An element of $(O,\ar{O})\ast O$ then consists of a composable arrangement of operations; that is, an operation of, say, arity $k$, together with $k$ operations for it to be composed with, one for each input.  The multiplication map $m$ then composes this arrangement to give a single operation in $O$.
\end{example}

\medskip

The monoidal category $[\cat{C},\cat{C}]_\cart$ admits a canonical (strict left) action $[\cat{C},\cat{C}]_\cart \times \cat{C}\longrightarrow \cat{C}$ on $\cat{C}$, given by evaluation.
In other words, this action is the transpose of the inclusion $[\cat{C},\cat{C}]_\cart\longrightarrow [\cat{C},\cat{C}]$.
We may precompose the strict monoidal forgetful functor $[\cat{C},\cat{C}]_\cart/T\longrightarrow [\cat{C},\cat{C}]_\cart$ with the above to obtain a (strict left) action
$[\cat{C},\cat{C}]_\cart/T\times\cat{C}\longrightarrow \cat{C}$.
Transporting this action along the monoidal equivalence \eqref{eqn:ev_1}, we obtain an action 
\[
\ast\colon \cat{C}/T1\times\cat{C}\longrightarrow \cat{C}
\]
(sometimes written as $\ast_T$). Concretely, the functor $\ast$ is defined by the pullback given in \eqref{eqn:tensor_collection};
note that the pullback $(P,\ar{P})\ast Q$ is independent of
the morphism $\ar{Q}$.
This is a \emph{pseudo} action, in the sense that
it is equipped with the canonical coherent isomorphisms
\[
I\ast C\cong C\qquad ((P,\ar{P})\otimes(Q,\ar{Q}))\ast C\cong 
(P,\ar{P})\ast((Q,\ar{Q})\ast C)
\]
natural in $(P,\ar{P}),(Q,\ar{Q})\in\cat{C}/T1$ and $C\in\cat{C}$. {The pseudo action $\ast$ is not strict (in general) because the monoidal equivalence \eqref{eqn:ev_1} is not a monoidal isomorphism (in general).}

Let $O=((O,\ar{O}),e,m)$ be a $T$-operad. An \defemph{$O$-algebra} is an object $A$ of $\cat{C}$ together with an action of $O$, i.e., a morphism $\alpha\colon (O,\ar{O})\ast A\longrightarrow A$ satisfying the usual axioms.
Note that via the monoidal equivalence \eqref{eqn:ev_1}, the $T$-operad $O$ corresponds to the cartesian monad $\overline{O}=O\ast(-)$ on $\cat{C}$ equipped with the cartesian monad morphism $\overline{\ar{O}}\colon O\ast(-)\longrightarrow T$.
An $O$-algebra is equivalent to an Eilenberg--Moore algebra for the monad $O\ast(-)$.
We say that the monad $O\ast(-)$ is \defemph{induced by} $O$.

\subsection{\texorpdfstring{$\omega$}{omega}-graphs and the free strict \texorpdfstring{$\omega$}{omega}-category monad \texorpdfstring{$T^{(\omega)}$}{T^(omega)}}

In order to define weak $\omega$-categories, we shall apply the above general theory of operads to the category $\enGph{\omega}$ of $\omega$-graphs and the monad $T^{(\omega)}$ for \emph{strict} $\omega$-category. 
These are obtained as limits of $\enGph{n}$ and $T^{(n)}$, which we now define.

\begin{definition}[\cite{Wolff_V-graph}]
\label{def:V-graph}
	For any locally small category $\cat{V}$, we define the category $\enGph{\cat{V}}$ as follows.
	\begin{itemize}
	\item An object is a (small) \defemph{$\cat{V}$-graph} $G=(\ob{G},(G(x,y))_{x,y\in\ob{G}})$, consisting of a small set $\ob{G}$ of \defemph{objects} and for each pair $x,y$ of objects, an object $G(x,y)$ of $\cat{V}$.
	\item A morphism $f$ from $G=(\ob{G},(G(x,y))_{x,y\in\ob{G}})$ to $G'=(\ob{G'},(G'(x,y))_{x,y\in\ob{G'}})$ consists of a function $\ob{f}\colon \ob{G}\longrightarrow\ob{G'}$ (whose action we denote by $f$) and, for each pair $x,y$ of objects of $G$, a morphism $f_{x,y}\colon G(x,y)\longrightarrow G'(fx,fy)$ in $\cat{V}$.\qedhere
	\end{itemize}
\end{definition}

The $\enGph{(-)}$ construction routinely extends to an endo-2-functor on the 2-category $\CAT$ of locally small categories.

\begin{definition}
\label{def:nGph}
	For each $n\in\N$, we define the category $\enGph{n}$ recursively as follows:
	\[
	\enGph{0}=\Set,\qquad \enGph{(n+1)}=\enGph{(\enGph{n})}.
	\]
	An object of $\enGph{n}$ is called an \defemph{$n$-graph}.
\end{definition}

\begin{remark}
\label{rmk:n-Gph_as_presheaf_cat}
	The $\enGph{(-)}$ construction preserves presheaf categories, i.e., for any small category $\cat{D}$, the category $\enGph{[\cat{D},\Set]}$ is equivalent to $[\cat{D}^+,\Set]$, where the category $\cat{D}^+$ is obtained from $\cat{D}$ by newly adding an object $\ast$ and, for each object $D\in\cat{D}$, two morphisms $s^{(D)},t^{(D)}\colon D\longrightarrow \ast$, such that for each morphism $f\colon D\longrightarrow D'$ in $\cat{D}$, $s^{(D')}\circ f=s^{(D)}$ and $ t^{(D')}\circ f=t^{(D)}$.
	Hence we see by induction that $\enGph{n}$ is equivalent to $[\G_n^\op,\Set]$, where $\G_n$ is the category freely generated by the graph 
	\[
\begin{tikzpicture}[baseline=-\the\dimexpr\fontdimen22\textfont2\relax ]
      \node(0) at (0,0) {$[0]$};
      \node(1) at (2,0) {$[1]$};
      \node(d) at (4,0) {$\cdots$};
      \node(n) at (6,0) {$[n]$};
      
      \draw [->,transform canvas={yshift=3pt}] (0) to node[auto, labelsize] 
      {$s_0$} (1); 
      \draw [->,transform canvas={yshift=-3pt}] (0) to node[auto, 
      swap,labelsize] 
      {$t_0$} (1); 
      \draw [->,transform canvas={yshift=3pt}] (1) to node[auto, labelsize] 
      {$s_1$} (d); 
      \draw [->,transform canvas={yshift=-3pt}] (1) to node[auto, 
      swap,labelsize] 
      {$t_1$} (d); 
      \draw [->,transform canvas={yshift=3pt}] (d) to node[auto, labelsize] 
      {$s_{n-1}$} (n); 
      \draw [->,transform canvas={yshift=-3pt}] (d) to node[auto, 
      swap,labelsize] 
      {$t_{n-1}$} (n); 
\end{tikzpicture}
\]
subject to the relations
\[
s_{k+1}\circ s_{k}=t_{k+1}\circ s_{k},\qquad s_{k+1}\circ t_{k} =t_{k+1}\circ t_{k} \qquad (k\in \{0,\dots,n-2\}).\qedhere
\]
\end{remark}

For any locally small category $\cat{V}$ with finite products, we have the category $\enCat{\cat{V}}$ of small $\cat{V}$-categories~\cite{Kelly_book}; throughout this paper, we only consider enrichment over cartesian $\cat{V}$.
The $\enCat{(-)}$ construction extends to an endofunctor on the category of locally small categories with finite products and finite product preserving functors.
($\enCat{(-)}$ is also a 2-functor, but we shall not use this fact.)
\begin{definition}
\label{def:nCat}
	For each $n\in\N$, we define the category $\enCat{n}$ recursively as follows:
	\[
	\enCat{0}=\Set,\qquad \enCat{(n+1)}=\enCat{(\enCat{n})}.
	\]
	An object of $\enCat{n}$ is called a \defemph{strict $n$-category}.
\end{definition}

Let $\cat{V}$ be a locally small category with finite products.
There is an evident forgetful functor $U^{(\cat{V})}\colon \enCat{\cat{V}}\longrightarrow\enGph{\cat{V}}$.
If $\cat{V}$ has small coproducts distributing over finite products, then $U^{(\cat{V})}$ admits a left adjoint $F^{(\cat{V})}$ \cite{Wolff_V-graph}. 
For each $n\in\N$, $\enCat{n}$ satisfies this condition (in fact, a stronger condition of \emph{extensivity} \cite{CFP_18}), so we have $F^{(\enCat{n})}\dashv U^{(\enCat{n})}$.

\begin{definition}\label{def:adj_Fn_Un}
For each $n\in\N$, we define the (monadic) adjunction 
$F^{(n)}\dashv U^{(n)}\colon \enCat{n}\longrightarrow\enGph{n}$ recursively as follows:
\begin{itemize}
	\item $F^{(0)}=U^{(0)}=\id_{\Set}$,
	\item $F^{(n+1)}\dashv U^{(n+1)}$ is the composite: 
	\[
	\begin{tikzpicture}[baseline=-\the\dimexpr\fontdimen22\textfont2\relax ]
      \node(0) at (0,0) {$\enGph{(\enGph{n})}$};
      \node(1) at (4,0) {$\enGph{(\enCat{n})}$};
      \node(2) at (8,0) {$\enCat{(\enCat{n})}$};
      
      \draw [->, transform canvas={yshift=5}] (0) to node[auto, labelsize] {$\enGph{F^{(n)}}$} (1); 
      \draw [->, transform canvas={yshift=5}] (1) to node[auto, labelsize] {$F^{(\enCat{n})}$} (2); 
      \draw [<-, transform canvas={yshift=-5}] (0) to node[auto,swap,labelsize] {$\enGph{U^{(n)}}$} (1); 
      \draw [<-, transform canvas={yshift=-5}] (1) to node[auto,swap,labelsize] {$U^{(\enCat{n})}$} (2); 
      
      \node [rotate=90] at (2,0) {$\vdash$}; 
      \node [rotate=90] at (6,0) {$\vdash$}; 
\end{tikzpicture}
	\]
	(note that $\enGph{(-)}$, being a 2-functor, preserves adjunctions).
\end{itemize}
We denote by $T^{(n)}$ the monad on $\enGph{n}$ induced by the adjunction $F^{(n)}\dashv U^{(n)}$. This monad is called the \defemph{free strict $n$-category monad}.
\end{definition}

\begin{proposition}[{\cite[Theorem F.2.1]{Leinster_book}, \cite[Theorem 4.6]{CFP_18}}]
For each $n\in\N$, the monad $T^{(n)}$ is cartesian.
\end{proposition}

\medskip

Let us now define $\enGph{\omega}$ and $T^{(\omega)}$ as limits of $\enGph{n}$ and $T^{(n)}$ respectively.
First, the category $\enGph{\omega}$ is the limit in $\CAT$ of the diagram
\begin{equation}
\label{eqn:limit_diagram_for_omega_graph}
\begin{tikzpicture}[baseline=-\the\dimexpr\fontdimen22\textfont2\relax ]
      \node(00) at (0,0) {$\enGph{0}$};
      \node(01) at (3,0) {$\enGph{1}$};
      \node(02) at (6,0) {$\enGph{2}$};
      \node(03) at (10,0) {$\cdots$.};
      
      \draw [<-] (00) to node[auto,labelsize] {$\ob{-}$} (01);  
      \draw [<-] (01) to node[auto,labelsize] {$\enGph{(\ob{-})}$} (02); 
      \draw [<-] (02) to node[auto,labelsize] {$\enGph{(\enGph{(\ob{-})})}$} (03); 
\end{tikzpicture}
\end{equation}
Objects of $\enGph{\omega}$ are called \defemph{$\omega$-graphs}.
\begin{remark}
\label{rmk:omega-Gph_as_presheaf_cat}
	By Remark~\ref{rmk:n-Gph_as_presheaf_cat}, the diagram \eqref{eqn:limit_diagram_for_omega_graph} is equivalent (in $\CAT^{(\omega^\op)}$) to
	\begin{equation}
	\label{eqn:limit_diagram_presheaf_cats}
\begin{tikzpicture}[baseline=-\the\dimexpr\fontdimen22\textfont2\relax ]
      \node(00) at (0,0) {$[\G_0^\op,\Set]$};
      \node(01) at (3,0) {$[\G_1^\op,\Set]$};
      \node(02) at (6,0) {$[\G_2^\op,\Set]$};
      \node(03) at (9,0) {$\cdots$,};
      
      \draw [<-] (00) to node[auto,labelsize] {$[J_0^\op,\Set]$} (01);  
      \draw [<-] (01) to node[auto,labelsize] {$[J_1^\op,\Set]$} (02); 
      \draw [<-] (02) to node[auto,labelsize] {$[J_2^\op,\Set]$} (03); 
\end{tikzpicture}
\end{equation}
where $J_n\colon \G_n\longrightarrow \G_{n+1}$ is the inclusion functor (mapping $[k]\in \G_n$ to $[k]\in \G_{n+1}$). 
Hence $\enGph{\omega}$ is equivalent to the limit of \eqref{eqn:limit_diagram_presheaf_cats}, which is again a presheaf category since $\lim_n [\G_n^\op,\Set]\cong [\colim_n \G_n^\op,\Set]$.
The category $\G=\colim_n \G_n$ is freely generated by the graph 
	\[
\begin{tikzpicture}[baseline=-\the\dimexpr\fontdimen22\textfont2\relax ]
      \node(0) at (0,0) {$[0]$};
      \node(1) at (2,0) {$[1]$};
      \node(d) at (4,0) {$\cdots$};
      \node(n) at (6,0) {$[n]$};
      \node(d2) at (8,0) {$\cdots$};
      
      \draw [->,transform canvas={yshift=3pt}] (0) to node[auto, labelsize] 
      {$s_0$} (1); 
      \draw [->,transform canvas={yshift=-3pt}] (0) to node[auto, 
      swap,labelsize] 
      {$t_0$} (1); 
      \draw [->,transform canvas={yshift=3pt}] (1) to node[auto, labelsize] 
      {$s_1$} (d); 
      \draw [->,transform canvas={yshift=-3pt}] (1) to node[auto, 
      swap,labelsize] 
      {$t_1$} (d); 
      \draw [->,transform canvas={yshift=3pt}] (d) to node[auto, labelsize] 
      {$s_{n-1}$} (n); 
      \draw [->,transform canvas={yshift=-3pt}] (d) to node[auto, 
      swap,labelsize] 
      {$t_{n-1}$} (n); 
      \draw [->,transform canvas={yshift=3pt}] (n) to node[auto, labelsize] 
      {$s_{n}$} (d2); 
      \draw [->,transform canvas={yshift=-3pt}] (n) to node[auto, 
      swap,labelsize] 
      {$t_{n}$} (d2);
\end{tikzpicture}
\]
subject to the relations
\[
s_{k+1}\circ s_{k} =t_{k+1}\circ s_{k},\qquad s_{k+1}\circ t_{k}=t_{k+1}\circ t_{k} \qquad (k\in \N).
\]
{Thus $\omega$-graphs are equivalent to presheaves over $\G$, which are sometimes called \emph{globular sets} \cite{Batanin_98,Leinster_book,Cheng_Leinster_coalg}.}
Let $G$ be an $\omega$-graph,  with the corresponding globular set $G'\colon \G^\op\longrightarrow \Set$. For each $k\in\N$, elements of the set $G'[k]$ are called \defemph{$k$-cells} of $G$. The functions $G's_k$ and $G't_k$ are written simply as $s_k$ and $t_k$, and called the ($k$-dimensional) \defemph{source} and \defemph{target} maps of $G$.
For any $(k+1)$-cell $f$ of $G$, we write $f\colon a\longrightarrow b$ to express that $s_k(f)=a$ and $t_k(f)=b$.
Henceforth we shall use the concepts of $\omega$-graph and globular set interchangeably.
\end{remark}

\begin{remark}
\label{rmk:omega_gph_as_terminal_coalg}
One can also give a concise \emph{coinductive} definition of $\omega$-graph as: an $\omega$-graph $G$ consists of a set $\ob{G}$ of objects and, for each pair $x,y$ of objects, an $\omega$-graph $G(x,y)$ \cite{vandenBerg_Garner}.
More precisely, the category $\enGph{\omega}$ is the carrier of the terminal coalgebra for $\enGph{(-)}$ (seen as an endofunctor on the category of locally small categories) \cite{Cheng_Leinster_coalg}. 
The structure map of this coalgebra is the functor $\enGph{\omega}\longrightarrow \enGph{(\enGph{\omega})}$ mapping each $\omega$-graph $G$ to the ($\enGph{\omega}$)-enriched graph consisting of the same objects and for each pair of objects $x$ and $y$ of $G$, the suitably defined hom $\omega$-graph $G(x,y)$.
By Lambek's lemma, this functor is an isomorphism of categories.
\end{remark}

Note that we have a monad in the 2-category $\CAT^{(\omega^\op)}$ on the object \eqref{eqn:limit_diagram_for_omega_graph} given by the sequence of cartesian monads $(T^{(n)})_{n\in\N}$.
Applying the 2-functor $\lim \colon \CAT^{(\omega^\op)}\longrightarrow \CAT$, we obtain a monad $T^{(\omega)}$ on $\enGph{\omega}$, called the \defemph{free strict $\omega$-category monad}. 
$T^{(\omega)}$ is also cartesian, since a commutative square in $\enGph{\omega}$ is a pullback if and only if for each $n\in\N$, it is mapped to a pullback in $\enGph{n}$ by the projection $\enGph{\omega}\longrightarrow \enGph{n}$.

As a consequence, we obtain the monoidal category $\enGph{\omega}/T^{(\omega)}1$, and monoids therein are \defemph{$T^{(\omega)}$-operads} (also called \emph{globular operads}).
{Logically, we may now proceed to the next step, but some expository comments might be helpful at this point. 
See \cite[Chapter 8]{Leinster_book} for a more detailed account.

Let us start with an explicit description of the $\omega$-graph $T^{(\omega)}1$. Using the free monoid monad (or the list monad) $(-)^\ast$ on $\Set$, $T^{(\omega)}1$ (regarded as a globular set) is given by the following diagram of sets:
\begin{equation}
	    \label{eqn:Tomega-globular-set}
\begin{tikzpicture}[baseline=-\the\dimexpr\fontdimen22\textfont2\relax ]
      \node(0) at (0,0) {$1$};
      \node(1) at (2,0) {$1^\ast$};
      \node(d) at (4,0) {$1^{\ast\ast}$};
      \node(n) at (6,0) {$1^{\ast\ast\ast}$};
      \node(d2) at (8,0) {$\cdots$};
      
      \draw [<-,transform canvas={yshift=3pt}] (0) to node[auto, labelsize] 
      {$s_0={!}$} (1); 
      \draw [<-,transform canvas={yshift=-3pt}] (0) to node[auto, 
      swap,labelsize] 
      {$t_0={!}$} (1); 
      \draw [<-,transform canvas={yshift=3pt}] (1) to node[auto, labelsize] 
      {$s_1={!^\ast}$} (d); 
      \draw [<-,transform canvas={yshift=-3pt}] (1) to node[auto, 
      swap,labelsize] 
      {$t_1={!^\ast}$} (d); 
      \draw [<-,transform canvas={yshift=3pt}] (d) to node[auto, labelsize] 
      {$s_{2}={!^{\ast\ast}}$} (n); 
      \draw [<-,transform canvas={yshift=-3pt}] (d) to node[auto, 
      swap,labelsize] 
      {$t_{2}={!^{\ast\ast}}$} (n); 
      \draw [<-,transform canvas={yshift=3pt}] (n) to node[auto, labelsize] 
      {$s_{3}={!^{\ast\ast\ast}}$} (d2); 
      \draw [<-,transform canvas={yshift=-3pt}] (n) to node[auto, 
      swap,labelsize] 
      {$t_{3}={!^{\ast\ast\ast}}$} (d2);
\end{tikzpicture}
\end{equation}
Here $1=\{\bullet\}$ is a singleton, $1^\ast$ is the set of lists of $\bullet$, $1^{\ast\ast}$ is the set of lists of lists of $\bullet$, and so on. It is important that the cells of $T^{(\omega)}1$ can be regarded as \emph{globular pasting schemes}; the following diagram shows some cells of $T^{(\omega)}1$ and the corresponding globular pasting schemes.
\begin{equation}
\label{eqn:globular_pasting_shcemes}
\begin{tikzpicture}[baseline=-\the\dimexpr\fontdimen22\textfont2\relax ]
      \node(21) at (0,0) {$\bullet$};
      \node(l1) at (0,-0.5) {$\bullet\in 1$};
\end{tikzpicture}
\qquad
\begin{tikzpicture}[baseline=-\the\dimexpr\fontdimen22\textfont2\relax ]
      \node(21) at (0,0) {$\bullet$};
      \node(l1) at (0,-0.5) {$[\,]\in 1^\ast$};
\end{tikzpicture}
\qquad
\begin{tikzpicture}[baseline=-\the\dimexpr\fontdimen22\textfont2\relax ]
      \node(31) at (-0.5,0) {$\bullet$};
      \node(32) at (0.75,0) {$\bullet$};
      \node(33) at (2,0) {$\bullet$};
      \node(l0) at (0.75,-0.5) {$[\bullet,\bullet] \in1^\ast$};
      
      \draw [->]  (31) to (32);
      \draw [->] (32) to  (33);
\end{tikzpicture}
\qquad
\begin{tikzpicture}[baseline=-\the\dimexpr\fontdimen22\textfont2\relax ]
      \node(20) at (-1.5,0) {$\bullet$};
      \node(21) at (0,0) {$\bullet$};
      \node(22) at (1.5,0) {$\bullet$};
      \node(23) at (3,0) {$\bullet$};
      \node(l0) at (0.75,-0.7) {$[[\bullet],[\,],[\bullet,\bullet]]\in 1^{\ast\ast}$};
      
      \draw [->]  (21) to (22);
      
      \draw [->,bend left=30]  (20) to node (2u) {} (21);      
      \draw [->,bend right=30] (20) to node (2b) {} (21); 
      \draw [->] (2u) to (2b);

      \draw [->,bend left=50]  (22) to node (3u) {} (23);
      \draw [->] (22) to node (3m) {} (23);
      \draw [->,bend right=50] (22) to node (3b) {} (23);
      \draw [->] (3u) to (3m);
      \draw [->] (3m) to (3b);
\end{tikzpicture}
\end{equation}

We now explain what a $T^{(\omega)}$-operad and its algebra amount to. Let $O=((O,\ar{O}),e,m)$ be a $T^{(\omega)}$-operad. As in the case of non-symmetric operads (Example~\ref{ex:non-sym-operad}), the cells of the $\omega$-graph $O$ can be regarded as operations. The morphism $\ar{O}\colon O\longrightarrow T^{(\omega)}1$ maps each operation to its arity, which is a globular pasting scheme. An $O$-algebra $(A,\alpha\colon (O,\ar{O})\ast A\longrightarrow A)$ consists of an $\omega$-graph $A$ equipped with the interpretation of each operation $\sigma$ in $O$ on it. For instance, if $\sigma$ is a $1$-cell of $O$ whose arity is $[\bullet,\bullet]$, then its interpretation on $A$ is an operation mapping each composable pair of $1$-cells in $A$ to a $1$-cell in $A$. In summary, $T^{(\omega)}$-operads form a notion of algebraic theory for $\omega$-graphs whose arities are the globular pasting schemes.}

\subsection{Contractions}
\label{subsec:contractions}
We now turn to the notion of contraction.
A contraction is a piece of structure on a morphism in $\enGph{\omega}$.
Leinster \cite{Leinster_book} introduced a set-theoretic definition of contraction, to which Garner \cite{Garner_univ} gave diagrammatic formulation. We adopt the latter.

In order to motivate the definition, we first review the classical notion of \emph{lifting property}.
Given morphisms $l\colon A\longrightarrow B$ and $r\colon C\longrightarrow D$ in a category $\cat{C}$, we say that $r$ has the \defemph{right lifting property} with respect to $l$ (or equivalently, $l$ has the \defemph{left lifting property} with respect to $r$) if, for any pair of morphisms $u\colon A\longrightarrow C$ and $v\colon B\longrightarrow D$ such that $v\circ l = r\circ u$, there exists a (not necessarily unique) $w\colon B\longrightarrow C$ making the diagram 
\begin{equation*}
\begin{tikzpicture}[baseline=-\the\dimexpr\fontdimen22\textfont2\relax ]
      \node(00) at (0,1) {$A$};
      \node(01) at (2,1) {$C$};
      \node(10) at (0,-1) {$B$};
      \node(11) at (2,-1) {$D$};
      
      \draw [->] (00) to node[auto, labelsize] {$u$} (01); 
      \draw [->] (01) to node[auto, labelsize] {$r$} (11); 
      \draw [->] (00) to node[auto,swap,labelsize] {$l$} (10); 
      \draw [->] (10) to node[auto,swap,labelsize] {$v$} (11);   
      \draw [->] (10) to node[midway,fill=white,labelsize] {$w$} (01);
\end{tikzpicture}
\end{equation*}
commute.

A contraction is an algebraic version of the right lifting property, given relative to a certain set $\cat{J}$ of morphisms. 

\begin{definition} [{\cite[Proposition 3.8]{Garner_understanding}}]
\label{def:contraction_general}
	Let $\cat{C}$ be a locally small category and $\cat{J}$ a set of morphisms in $\cat{C}$. 

	A \defemph{contraction} (with respect to $\cat{J}$) on a morphism $r\colon C\longrightarrow D$ in $\cat{C}$ is a function $\kappa$ assigning, for each element $l\colon A\longrightarrow B$ in $\cat{J}$ and each $u\colon A\longrightarrow C$ and $v\colon B\longrightarrow D$ such that $r\circ u=v\circ l$, a morphism $\kappa(l,u,v)\colon B\longrightarrow C$ such that $u=\kappa(l,u,v)\circ l$ and $v=r\circ \kappa(l,u,v)$.

Given morphisms $r\colon C\longrightarrow D$ and $r'\colon C'\longrightarrow D'$ equipped with contractions $\kappa$ and $\kappa'$ respectively, a map of morphisms $(h\colon C\longrightarrow C', k\colon D\longrightarrow D' )\colon r\longrightarrow r'$ (i.e., a commutative square)
is said to \defemph{preserve contractions} if for each $(l,u,v)$ in the domain of $\kappa$, $h\circ \kappa(l,u,v)= \kappa' (l,h\circ u,k\circ v)$.
\begin{equation*}
\begin{tikzpicture}[baseline=-\the\dimexpr\fontdimen22\textfont2\relax ]
      \node(00) at (0,2) {$A$};
      \node(01) at (2,2) {$C$};
      \node(10) at (0,0) {$B$};
      \node(11) at (2,0) {$D$};
      \node(02) at (5,2) {$C'$};
      \node(12) at (5,0) {$D'$};
      
      \draw [->] (00) to node[auto, labelsize] {$u$} (01); 
      \draw [->] (01) to node[auto, labelsize] {$h$} (02); 
      \draw [->] (01) to node[auto,near start, labelsize] {$r$} (11); 
      \draw [->] (02) to node[auto, labelsize] {$r'$} (12); 
      \draw [->] (11) to node[auto,swap,labelsize] {$k$} (12);  
      \draw [->] (00) to node[auto,swap,labelsize] {$l$} (10); 
      \draw [->] (10) to node[auto,swap,labelsize] {$v$} (11);   
      \draw [->] (10) to node[midway,fill=white,labelsize] {$\kappa(l,u,v)$} (01);
      \draw [->] (10) to node[near end,fill=white,labelsize] {$\kappa'(l,h\circ u,k\circ v)$} (02);
\end{tikzpicture}\qedhere
\end{equation*}
\end{definition}

The category of morphisms in $\cat{C}$ equipped with contractions with respect to $\cat{J}$ and contraction preserving maps is denoted by $\Contr(\cat{C},\cat{J})$.
In fact, we shall be mainly interested in certain subcategories of $\Contr(\cat{C},\cat{J})$, defined as follows.
{Denote the evident codomain functor by $\mathrm{cod}\colon \Contr(\cat{C},\cat{J})\longrightarrow\cat{C}$. Then for any object $D$ of $\cat{C}$, let $\Contr(\cat{C},\cat{J})_D$ be the fibre of $\mathrm{cod}$ over $D$;}
so an object of $\Contr(\cat{C},\cat{J})_D$ is a morphism $r$ in $\cat{C}$ with codomain $D$ equipped with a contraction, and a morphism is a contraction preserving map whose second component is $\id_D$. 
Note that there is a forgetful functor $\Contr(\cat{C},\cat{J})_D\longrightarrow \cat{C}/D$.

\begin{remark}
\label{rmk:Contr_G_fibration}
If $\cat{C}$ has pullbacks, then for each set $\cat{J}$ of morphisms in $\cat{C}$, the functor $\mathrm{cod}\colon \Contr(\cat{C},\cat{J})\longrightarrow \cat{C}$ is a (Grothendieck) fibration. Indeed, given any morphism $k\colon D\longrightarrow D'$ in $\cat{C}$ and any object $(r'\colon C'\longrightarrow D',\kappa')\in \Contr(\cat{C},\cat{J})_{D'}$, one can endow the pullback $k^\ast r'$ as in
\[
\begin{tikzpicture}[baseline=-\the\dimexpr\fontdimen22\textfont2\relax ]
      \node(01) at (2,1) {$k^\ast C'$};
      \node(11) at (2,-1) {$D$};
      \node(20) at (4,1) {$C'$};
      \node(21) at (4,-1) {$D'$};
      
      \draw [->] (01) to node[auto, swap,labelsize] {$k^\ast r'$} (11); 
      \draw [->] (01) to node[auto, labelsize] {$h$} (20);
      \draw [->] (11) to node[auto,swap, labelsize] {$k$} (21);
      \draw [->] (20) to node[auto, labelsize] {$r'$} (21);
            \draw (2.3,0.2) to (2.8,0.2) to (2.8,0.7);
\end{tikzpicture}
\]
with a contraction $k^\ast \kappa'$, induced from $\kappa'$ by the universality of pullback. 
Then the morphism 
\[
(h,k)\colon (k^\ast r',k^\ast \kappa)\longrightarrow (r',\kappa')
\] 
in $\Contr(\cat{C},\cat{J})$ is the required cartesian lifting of $k$.
\end{remark}

In order to define Leinster's notion of contraction for morphisms in $\enGph{\omega}$, we define a set $\cat{J}^{(\omega)}$ of morphisms in $\enGph{\omega}$ (called the set of \emph{generating cofibrations} in \cite{Garner_univ}).
Recall from Remark \ref{rmk:omega-Gph_as_presheaf_cat} the equivalence $\enGph{\omega}\simeq [\G^\op,\Set]$.
We denote the Yoneda embedding by $\Yoneda\colon \G\longrightarrow [\G^\op,\Set]$.
The set $\cat{J}^{(\omega)}$ is defined to be 
$\{\, m_k\colon \partial \Yoneda [k]\longrightarrow \Yoneda [k]\mid k\in\N\,\}$, where $\partial \Yoneda [k]$ is the subobject of $\Yoneda [k]$ obtained by removing the unique $k$-cell $\id_{[k]}$ of $\Yoneda [k]$, and $m_k$ is the associated inclusion.
These morphisms may be depicted as follows:
\[
\mathcal{J}^{(\omega)}=\left\{
\begin{tikzpicture}[baseline=-\the\dimexpr\fontdimen22\textfont2\relax ]
      \node(11) at (0,1) {$\bigg( \quad\bigg) $};
      \node(21) at (0,-1) {$\bigg(\bullet\bigg)$};
      
      \draw [->] (0, 0.4) to node[auto,labelsize]{$m_0$} (0,-0.4);
\end{tikzpicture}, 
\begin{tikzpicture}[baseline=-\the\dimexpr\fontdimen22\textfont2\relax ]
      \node(11) at (0,1) {$\bigg( \bullet$};
      \node(12) at (1.5,1) {$\bullet \bigg)$};
      \node(21) at (0,-1) {$\bigg( \bullet$};
      \node(22) at (1.5,-1) {$\bullet \bigg)$};
      
      \draw [->]  (21) to(22);      
      
      \draw [->] (0.75, 0.4) to  node[auto,labelsize]{$m_1$} (0.75,-0.4);
\end{tikzpicture}, 
\begin{tikzpicture}[baseline=-\the\dimexpr\fontdimen22\textfont2\relax ]
      \node(11) at (0,1) {$\bigg( \bullet$};
      \node(12) at (1.5,1) {$\bullet \bigg)$};
      \node(21) at (0,-1) {$\bigg( \bullet$};
      \node(22) at (1.5,-1) {$\bullet \bigg)$};
      
      \draw [->,bend left=30]  (11) to node (1u) {} (12);
      \draw [->,bend right=30] (11) to node (1b) {} (12);
      \draw [->,bend left=30]  (21) to node (2u) {} (22);      
      \draw [->,bend right=30] (21) to node (2b) {} (22); 
      
      \draw [->] (2u) to (2b);
      
      \draw [->] (0.75, 0.4) to  node[auto,labelsize]{$m_2$} (0.75,-0.4);
\end{tikzpicture}, 
\begin{tikzpicture}[baseline=-\the\dimexpr\fontdimen22\textfont2\relax ]
      \node(11) at (0,1) {$\bigg( \bullet$};
      \node(12) at (1.5,1) {$\bullet \bigg)$};
      \node(21) at (0,-1) {$\bigg( \bullet$};
      \node(22) at (1.5,-1) {$\bullet \bigg)$};
      
      \draw [->,bend left=30]  (11) to node (1u) {} (12);
      \draw [->,bend right=30] (11) to node (1b) {} (12);
      \draw [->,bend left=30]  (21) to node (2u) {} (22);      
      \draw [->,bend right=30] (21) to node (2b) {} (22); 
      
      \draw [->,transform canvas={xshift=-0.45em}, bend right=30] (1u) to (1b);
      \draw [->,transform canvas={xshift=0.45em}, bend left=30]  (1u) to (1b);
      \draw [->,transform canvas={xshift=-0.45em}, bend right=30] (2u) to node (3s) {} (2b);
      \draw [->,transform canvas={xshift=0.45em}, bend left=30]  (2u) to node (3t) {} (2b);
      
      \draw [->] (0.6,-1) to (0.9,-1);
      \draw [->] (0.75, 0.4) to  node[auto,labelsize]{$m_3$} (0.75,-0.4);
\end{tikzpicture}, \ \dots
\right\}.
\] 
The geometric idea is that $\Yoneda [k]$ is the $\omega$-graph representing the (directed) \emph{$k$-dimensional disc} and $\partial \Yoneda [k]$ is its {boundary}, the (directed) \emph{$(k-1)$-dimensional sphere}. 
Note that for any $\omega$-graph $G$ and $k\in\N$, a morphism $\Yoneda [k]\longrightarrow G$ corresponds to a $k$-cell of $G$ by the Yoneda lemma.
{Similarly, for $k\geq 1$, a morphism $\partial \Yoneda [k]\longrightarrow G$ corresponds to a \emph{parallel pair of $(k-1)$-cells} of $G$.
Here, two $(k-1)$-cells ($k\geq 2$) $a$ and $b$ are said to be \defemph{parallel} if $s_{k-2}a=s_{k-2}b$ and $t_{k-2}a=t_{k-2}b$ hold, and we count any two $0$-cells as parallel. (We may formally extend this correspondence to the case where $k=0$ by adopting the convention that in any $\omega$-graph there is precisely one ``parallel pair of $(-1)$-cells''.)}

By a \defemph{contraction} on a morphism $r\colon C\longrightarrow D$ in $\enGph{\omega}$ we always mean a contraction (in the sense of Definition \ref{def:contraction_general}) with respect to $\cat{J}^{(\omega)}$. 
So such a contraction $\kappa$ assigns for each $k\in\N$, each pair $c,c'$ of parallel $(k-1)$-cells of $C$ and each $k$-cell $d\colon r(c) \longrightarrow r(c')$ in $D$, a $k$-cell $\kappa(k,(c,c'), d)\colon c\longrightarrow c'$ in $C$ such that $d=r(\kappa (k, (c,c'),d))$.
For any $\omega$-graph $D$, we write the category $\Contr(\enGph{\omega},\cat{J}^{(\omega)})_D$ simply as $\Contr_D$.

\subsection{The \texorpdfstring{$T^{(\omega)}$}{T^(omega)}-operad \texorpdfstring{$L$}{L} for weak \texorpdfstring{$\omega$}{omega}-categories}
\label{subsec:Ln}
We define the category $\OC{\omega}{T^{(\omega)}}$ of \defemph{$T^{(\omega)}$-operads with contractions} as the following pullback of categories (where the arrows to $\enGph{\omega}/T^{(\omega)}1$ denote the forgetful functors):
\begin{equation}
\label{eqn:OC_as_pullback}
\begin{tikzpicture}[baseline=-\the\dimexpr\fontdimen22\textfont2\relax ]
      \node(00) at (0,1) {$\OC{\omega}{T^{(\omega)}}$};
      \node(01) at (4,1) {$\Mon{\enGph{\omega}/T^{(\omega)}1}$};
      \node(10) at (0,-1) {$\Contr_{T^{(\omega)}1}$};
      \node(11) at (4,-1) {$\enGph{\omega}/T^{(\omega)}1$.};
      
      \draw [->] (00) to node[auto, labelsize] {} (01); 
      \draw [->] (01) to node[auto, labelsize] {} (11); 
      \draw [->] (00) to node[auto,swap,labelsize] {} (10); 
      \draw [->] (10) to node[auto,swap,labelsize] {} (11); 

      \draw (0.3,0.2) to (0.8,0.2) to (0.8,0.7);
\end{tikzpicture}
\end{equation}

{Roughly speaking, a contraction on a $T^{(\omega)}$-operad generates both (unbiased) composition operations and operations which yield coherence cells; see \cite[Chapter~9]{Leinster_book} for a detailed discussion. We shall use the universal $T^{(\omega)}$-operad with a contraction for our definition of weak $\omega$-category.}

\begin{proposition}[{\cite[Proposition 9.2.2]{Leinster_book}}]
	The category $\OC{\omega}{T^{(\omega)}}$ has an initial object.
\end{proposition}
Let $L$ be the initial object in $\OC{\omega}{T^{(\omega)}}$; we denote the $T^{(\omega)}$-operad underlying $L$ also by $L$.
\begin{definition}[{\cite[Definition 9.2.3]{Leinster_book}}]
	A \defemph{weak $\omega$-category} is an $L$-algebra.
\end{definition}

We denote the Eilenberg--Moore category of the monad $L\ast(-)$ by $\WkCats{\omega}$; the morphisms in $\WkCats{\omega}$ are called the \defemph{strict $\omega$-functors}, hence the subscript `s'.

\begin{example}
A canonical source of examples of weak $\omega$-categories is provided by \emph{algebras of a contractible $T^{(\omega)}$-operad} \cite[Example 9.2.4]{Leinster_book}.
Here we say that a $T^{(\omega)}$-operad $O=((O,\ar{O}),e,m)$ is \defemph{contractible} if it admits \emph{some} contraction; or equivalently, if the morphism $\ar{O}\colon O\longrightarrow T^{(\omega)}1$ has the right lifting property with respect to each $m_k\in \cat{J}^{(\omega)}$. Given such a $T^{(\omega)}$-operad $O$ and a choice of a contraction $\kappa$ on it, we obtain the unique morphism $\phi\colon L\longrightarrow O$ in $\OC{}{T^{(\omega)}}$ by the initiality of $L$. It then induces a monad morphism $\phi\ast(-)\colon L\ast (-)\longrightarrow O\ast(-)$, hence in turn a functor $\Alg{(O\ast(-))}\longrightarrow \WkCats{\omega}$. So any $O$-algebra, together with a choice of a contraction on $O$, gives rise to a weak $\omega$-category.

For example, the terminal $T^{(\omega)}$-operad, whose arity map is just $\id\colon T^{(\omega)}1\longrightarrow T^{(\omega)}1$, admits a unique contraction, hence any algebra for it---which is just an Eilenberg--Moore algebra for the monad $T^{(\omega)}$, i.e., a \emph{strict} $\omega$-category---is canonically a weak $\omega$-category.
As a less trivial example, Leinster constructs the \emph{fundamental weak $\omega$-category} of a topological space $X$ by exhibiting an action of a contractible $T^{(\omega)}$-operad on the $\omega$-graph consisting of higher homotopies in $X$; see \cite[Example 9.2.7]{Leinster_book}.

One can also define the notion of algebra of a $T^{(\omega)}$-operad over more general categories than $\enGph{\omega}$ \cite{Batanin_98,Lumsdaine}, i.e., internally in those categories.
In \cite{vandenBerg_Garner,Lumsdaine}, such a notion is defined on suitable categories of globular objects in the classifying (or syntactic) category $\mathcal{C}l(\mathbb{T})$ (cf.~\cite[Section 6]{Pitts_cat_log}) of a Martin-L\"of type theory  $\mathbb{T}$, and it is shown (using iterated identity types) that each type in $\mathbb{T}$ admits an action of a contractible $T^{(\omega)}$-operad, hence an internal weak $\omega$-category structure.
\end{example}

\section{The forgetful functor \texorpdfstring{$\Us\colon \WkCats{\omega}\longrightarrow \enGph{(\WkCats{\omega})}$}{Us}}
\label{sec:forgetful_strict}
In this section, we define the forgetful functor $\Us$,
inducing a \emph{hom weak $\omega$-category} $\cat{A}(x,y)$ over each weak $\omega$-category $\cat{A}$ and pair of objects $x,y\in\cat{A}$.
{As we shall see, our construction of $\Us$ heavily depends on the fact that $\enGph{(-)}$ preserves a lot of structure.}
{First we observe} that both the domain and codomain of $\Us$ are monadic over $\enGph{\omega}$; that the codomain is so is a consequence of the following.

\begin{proposition}[{Cf.~\cite[Proposition F.1.1 (b)]{Leinster_book}}]
\label{prop:Gph_preserves_EM}
	The 2-functor $\enGph{(-)}\colon \CAT\longrightarrow \CAT$ preserves Eilenberg--Moore objects. 
	That is, for each monad $T$ on a locally small category $\cat{V}$, the canonical comparison functor $\enGph{(\Alg{T})}\longrightarrow \Alg{(\enGph{T})}$ is an isomorphism of categories.
\end{proposition}
{\begin{proof}
	The monad $\enGph{T}$ is on the category $\enGph{\cat{V}}$.
	An object of $\Alg{(\enGph{T})}$ consists of a $\cat{V}$-graph $G$ together with a $\cat{V}$-graph morphism $\gamma\colon (\enGph{T})G\longrightarrow G$ satisfying the axioms of Eilenberg--Moore algebra. 
	The unit axiom forces $\gamma$ to be the identity on objects, so such a $\gamma$ consists of, for each pair $(x,y)$ of objects of $G$, a morphism $\gamma_{x,y}\colon TG(x,y)\longrightarrow G(x,y)$ in $\cat{V}$ satisfying the Eilenberg--Moore axioms. These data amount to give an Eilenberg--Moore algebra structure on each $G(x,y)$, and hence correspond to a $(\Alg{T})$-graph.
\end{proof}}

So $\enGph{(\WkCats{\omega})}=\enGph{(\Alg{(L\ast_{T^{(\omega)}}(-))})}$ is isomorphic to $\Alg{(\enGph{(L\ast_{T^{(\omega)}}(-))})}$. The monad $\enGph{(L\ast_{T^{(\omega)}}(-))}$ is on $\enGph{(\enGph{\omega})}$, which is canonically isomorphic to $\enGph{\omega}$ by Remark \ref{rmk:omega_gph_as_terminal_coalg}.
{Explicitly, the functor part of the monad $\enGph{(L\ast_{T^{(\omega)}}(-))}$ maps an $\omega$-graph $G$ to the $\omega$-graph $\enGph{(L\ast_{T^{(\omega)}}(-))}(G)$ with the same objects and such that for each pair $(x,y)$ of objects, the hom $(\enGph{(L\ast_{T^{(\omega)}}(-))}(G))(x,y)$ is equal to $L\ast_{T^{(\omega)}}(G(x,y))$, the underlying $\omega$-graph of the free weak $\omega$-category over the hom $\omega$-graph $G(x,y)$ of $G$.}
We shall induce $\Us$ from a monad morphism, that is induced by initiality of $L$.

The monad $\enGph{(L\ast_{T^{(\omega)}}(-))}$ is also induced from an operad over a cartesian monad. To show this, we use the following fact.
\begin{proposition}
\label{prop:Gph_preserves_cartesianness}
\begin{enumerate}
	\item Let $\cat{V}$ be a locally small category with pullbacks (resp.~finite limits). Then the category $\enGph{\cat{V}}$ has pullbacks (resp.~finite limits).
	\item Let $\cat{V}$ and $\cat{W}$ be locally small categories with pullbacks and $F\colon \cat{V}\longrightarrow \cat{W}$ be a pullback preserving functor. 
	Then the functor $\enGph{F}\colon\enGph{\cat{V}}\longrightarrow \enGph{\cat{W}}$ preserves pullbacks.
	\item Let $\cat{V}$  be a locally small category, $\cat{W}$ be a locally small category with pullbacks, $F,G\colon \cat{V}\longrightarrow \cat{W}$ be functors and $\alpha\colon F\longrightarrow G$ be a cartesian natural transformation. 
   Then the natural transformation $\enGph{\alpha}\colon\enGph{F}\longrightarrow \enGph{G}$ is cartesian. 
\end{enumerate}
\end{proposition}

The $T^{(\omega)}$-operad $L$ corresponds to the cartesian monad morphism $\overline{\ar{L}}\colon L\ast_{T^{(\omega)}}(-)\longrightarrow T^{(\omega)}$. {By Proposition~\ref{prop:Gph_preserves_cartesianness}, $\enGph{(-)}$ preserves cartesian monads as well as cartesian monad morphisms. Hence $\enGph{T^{(\omega)}}$ is a cartesian monad and $\enGph{\overline{\ar{L}}}\colon \enGph{(L\ast_{T^{(\omega)}}(-))}\longrightarrow \enGph{T^{(\omega)}}$ is a cartesian monad morphism. Since a cartesian monad morphism to $\enGph{T^{(\omega)}}$ corresponds to a $(\enGph{T^{(\omega)}})$-operad, it follows that the monad $\enGph{(L\ast_{T^{(\omega)}}(-))}$ is induced from a $(\enGph{T^{(\omega)}})$-operad.}

{$(\enGph{T^{(\omega)}})$-operads also form a notion of algebraic theory for $\omega$-graphs, but their arities are more restricted than those of $T^{(\omega)}$-operads. Here is an explicit description of the $\omega$-graph $(\enGph{T^{(\omega)}})1$ (seen as a globular set):
	\[
\begin{tikzpicture}[baseline=-\the\dimexpr\fontdimen22\textfont2\relax ]
      \node(0) at (0,0) {$1$};
      \node(1) at (2,0) {$1$};
      \node(d) at (4,0) {$1^{\ast}$};
      \node(n) at (6,0) {$1^{\ast\ast}$};
      \node(d2) at (8,0) {$\cdots$};
      
      \draw [<-,transform canvas={yshift=3pt}] (0) to node[auto, labelsize] 
      {$s_0={\id}$} (1); 
      \draw [<-,transform canvas={yshift=-3pt}] (0) to node[auto, 
      swap,labelsize] 
      {$t_0={\id}$} (1); 
      \draw [<-,transform canvas={yshift=3pt}] (1) to node[auto, labelsize] 
      {$s_1={!}$} (d); 
      \draw [<-,transform canvas={yshift=-3pt}] (1) to node[auto, 
      swap,labelsize] 
      {$t_1={!}$} (d); 
      \draw [<-,transform canvas={yshift=3pt}] (d) to node[auto, labelsize] 
      {$s_{2}={!^{\ast}}$} (n); 
      \draw [<-,transform canvas={yshift=-3pt}] (d) to node[auto, 
      swap,labelsize] 
      {$t_{2}={!^{\ast}}$} (n); 
      \draw [<-,transform canvas={yshift=3pt}] (n) to node[auto, labelsize] 
      {$s_{3}={!^{\ast\ast}}$} (d2); 
      \draw [<-,transform canvas={yshift=-3pt}] (n) to node[auto, 
      swap,labelsize] 
      {$t_{3}={!^{\ast\ast}}$} (d2);
\end{tikzpicture}
\]
The cells of $(\enGph{T^{(\omega)}})1$ represent the globular pasting schemes \emph{not involving compositions along $0$-cells}. For instance, the $3$-cell $[[\bullet],[\,],[\bullet,\bullet]]\in 1^{\ast\ast}$ of $(\enGph{T^{(\omega)}})1$ corresponds to the following globular pasting scheme (cf.~\eqref{eqn:globular_pasting_shcemes}):
\[
\begin{tikzpicture}[baseline=-\the\dimexpr\fontdimen22\textfont2\relax ]
      \node(21) at (0,0) {$\bullet$};
      \node(22) at (2,0) {$\bullet$};
      
      \draw [->,bend left=70,distance=30]  (21) to node(1) {} (22);
      \draw [->,bend left=25]  (21) to node (2) {} (22);
      \draw [->,bend left=-25]  (21) to node (3) {} (22);
      \draw [->,bend left=-70,distance=30]  (21) to node (4) {} (22);
      
      \draw [->,transform canvas={xshift=-0.5em}, bend right=30]  (1) to  (2);      
      \draw [->,transform canvas={xshift=0.5em}, bend left=30]  (1) to  (2);   
      \draw [->] (2) to (3);
      \draw [->,transform canvas={xshift=-1em,yshift=0.2em}, bend right=30]  (3) to (4);   
      \draw [->] (3) to  (4);
      \draw [->,transform canvas={xshift=1em,yshift=0.2em}, bend left=30]  (3) to  (4);  

      \draw [->] (0.8,0.65) to (1.2,0.65);
      \draw [->] (0.6,-0.6) to (0.9,-0.6);
      \draw [->] (1.1,-0.6) to (1.4,-0.6);
\end{tikzpicture}
\]
}

We turn to a concrete description of the $(\enGph{T^{(\omega)}})$-operad {corresponding to $\enGph{\overline{\ar{L}}}$, which induces} $\enGph{(L\ast_{T^{(\omega)}}(-))}$. For any locally small category $\cat{V}$, define the functor $\nameof{-}_\cat{V}\colon \cat{V}\longrightarrow \enGph{\cat{V}}$ (also written as $\nameof{-}$) by mapping each object $X\in \cat{V}$ to the $\cat{V}$-graph $\nameof{X}$ with a single object $\ast$ such that $\nameof{X}(\ast,\ast) = X$.
{Notice that $\nameof{-}$ preserves the terminal object when $\cat{V}$ has one. We shall take advantage of this fact and denote the terminal object of $\enGph{\cat{V}}$ by $\nameof{1}$, distinguishing it from the terminal object $1$ of $\cat{V}$. The terminal object of $\enGph{\omega}\cong \enGph{(\enGph{\omega})}$ is denoted by $1$. }

\begin{proposition}
\label{prop:Gph_and_nameof}
	Let $\cat{V}$ be a locally small category with finite limits and $T$ be a cartesian monad on $\cat{V}$.
	The functor 
	$\nameof{-}_T\colon\cat{V}/T1\longrightarrow \enGph{\cat{V}}/(\enGph{T})\nameof{1}$
	mapping $(\ar{P}\colon P\longrightarrow T1)$ to 
	\[
		\begin{tikzpicture}[baseline=-\the\dimexpr\fontdimen22\textfont2\relax ]
      \node(0) at (0,0) {$\nameof{P}$};
      \node(1) at (2,0) {$\nameof{T1}$};
      \node(2) at (4,0) {$(\enGph T)\nameof{1}$};
      
      \draw [->] (0) to node[auto, labelsize] {$\nameof{\ar{P}}$} (1); 
      \draw [->] (1) to node[auto, labelsize] {$\cong$} (2);  
\end{tikzpicture}
\]
makes the following square commute up to a natural isomorphism:
\[
\begin{tikzpicture}[baseline=-\the\dimexpr\fontdimen22\textfont2\relax ]
      \node(00) at (0,0.5) {$[\cat{V},\cat{V}]_\cart/T$};
      \node(01) at (5.5,0.5) {$\cat{V}/T1$};
      \node(10) at (0,-1) {$[\enGph{\cat{V}},\enGph{\cat{V}}]_\cart/\enGph{T}$};
      \node(11) at (5.5,-1) {$\enGph{\cat{V}}/(\enGph{T})\nameof{1}$.};
      
      \draw [->] (00) to node[auto, labelsize]{$\ev_1$} node[auto,swap,labelsize] {$\simeq$} (01); 
      \draw [->] (01) to node[auto, labelsize] {$\nameof{-}_T$} (11); 
      \draw [->] (00) to node[auto,swap,labelsize] {$\enGph{(-)}$} (10); 
      \draw [->] (10) to node [auto,labelsize] {$\ev_1$}node[auto,swap,labelsize] {$\simeq$} (11); 
\end{tikzpicture}
\] 
\end{proposition}

It follows that the functor $\nameof{-}_T$ acquires the structure of a strong monoidal functor, since the functor $\enGph{(-)}\colon [\cat{V},\cat{V}]_\cart/T\longrightarrow [\enGph{\cat{V}},\enGph{\cat{V}}]_\cart/\enGph{T}$ does.

{
When $\cat{V}=\enGph{\omega}$, the functor $\nameof{-}_{\enGph{\omega}}\colon \enGph{\omega}\longrightarrow \enGph{(\enGph{\omega})}\cong \enGph{\omega}$ maps an $\omega$-graph $G$ to the $\omega$-graph $\nameof{G}$ with a single $0$-cell $\ast$ and in which a $(k+1)$-cell is given by a $k$-cell of $G$ for all $k\in\mathbb{N}$. In view of the equivalence $\enGph{\omega}\simeq [\G^\op,\Set]$, we have the evident functor $S\colon \G\longrightarrow \G$ mapping $[k]$ to $[k+1]$, and $\nameof{-}_{\enGph{\omega}}$ is the right Kan extension along $S^\op$.
Observe that $(\enGph{T^{(\omega)}})1\cong \nameof{T^{(\omega)}1}_{\enGph{\omega}}$ is obtained from $T^{(\omega)}1$ by this construction.}

As a special case {of $\nameof{-}_T$}, we obtain 
\begin{equation}
\label{eqn:nameof_between_collection}
\nameof{-}_{T^{(\omega)}}\colon \enGph{\omega}/T^{(\omega)}1\longrightarrow \enGph{\omega}/(\enGph{T^{(\omega)}})1;
\end{equation}
here again we are identifying the two canonically isomorphic categories $\enGph{\omega}$ and $\enGph{(\enGph{\omega})}$. 
{Given an $\omega$-graph over $T^{(\omega)}1$, $\nameof{-}_{T^{(\omega)}}$ raises the dimensions of cells by one.}
Since $\nameof{-}_{T^{(\omega)}}$ is strong monoidal, we also obtain 
\begin{equation}
\label{eqn:nameof_between_operad}
\Mon{\nameof{-}_{T^{(\omega)}}}\colon \Mon{\enGph{\omega}/T^{(\omega)}1}\longrightarrow \Mon{\enGph{\omega}/(\enGph{T^{(\omega)}})1}.
\end{equation}

Here are some properties of $\nameof{-}$.
For any locally small category $\cat{V}$, $\nameof{-}_\cat{V}$ is fully faithful, and a $\cat{V}$-graph is in the essential image of $\nameof{-}_\cat{V}$ if and only if it has precisely one object. 
Accordingly, for any locally small category $\cat{V}$ with finite limits and a cartesian monad $T$ thereon, $\nameof{-}_T$ is also fully faithful, and an object $(\ar{G}\colon G\longrightarrow (\enGph{T})\nameof{1})\in \enGph{\cat{V}}/(\enGph{T})\nameof{1}$ is in the essential image of $\nameof{-}_T$ if and only if $G\in \enGph{\cat{V}}$ has precisely one object.

For a locally small category $\cat{V}$ with small coproducts, the functor $\nameof{-}_\cat{V}\colon \cat{V}\longrightarrow\enGph{\cat{V}}$ admits a left adjoint $\coprod_\cat{V}\colon\enGph{\cat{V}}\longrightarrow \cat{V}$, mapping $G\in\enGph{\cat{V}}$ to $\coprod_{x,y\in\ob{G}}G(x,y)\in\cat{V}$.
In particular, the functor
$\nameof{-}_\enGph{\omega}\colon \enGph{\omega}\longrightarrow \enGph{\omega}$
admits a left adjoint $\coprod_\enGph{\omega}\colon\enGph{\omega}\longrightarrow \enGph{\omega}$. 
Intuitively, given an $\omega$-graph $G$, the $\omega$-graph $\coprod_{\enGph{\omega}}G$ is obtained by {\emph{lowering} the dimensions of cells} by one: for each $k\geq 1$, a $k$-cell of $G$ is turned to a $(k-1)$-cell of $\coprod_{\enGph{\omega}}G$, and the 0-cells of $G$ are thrown away. In view of the equivalence $\enGph{\omega}\simeq [\G^\op,\Set]$, $\coprod_{\enGph{\omega}}$ can be seen as the precomposition of $S^\op$.
Recall the set $\cat{J}^{(\omega)}=\{\,m_k\mid k\in\N\,\}$ of morphisms in $\enGph{\omega}$.
We have $\coprod m_{k+1}\cong m_k$ for each $k\in \N$, whereas $\coprod m_0$ is the identity morphism on the empty (= initial) $\omega$-graph.
In the sequel, for simplicity we identify $\coprod m_{k+1}$ with $m_k$.

\begin{proposition}
\label{prop:nameof_contraction}
	For any morphism $r\colon C\longrightarrow D$ in $\enGph{\omega}$, the adjunction $\coprod\dashv \nameof{-}$ yields a canonical bijective correspondence between contractions on $r$ and contractions on $\nameof{r}$. In more detail, a contraction $\kappa$ on $r$ corresponds to a contraction $\nameof{\kappa}$ on $\nameof{r}$ if and only if whenever $\widehat{u}$ is the transpose of $u$ and $\widehat{v}$ is the transpose of $v$ in the outer commutative squares in \eqref{eqn:contraction_adjoint}, $\kappa(k,u,v)$ is the transpose of $\nameof{\kappa}(k+1,\widehat{u},\widehat{v})$.
\begin{equation}
\label{eqn:contraction_adjoint}
\begin{tikzpicture}[baseline=-\the\dimexpr\fontdimen22\textfont2\relax ]
      \node(00) at (0,1) {$\partial \Yoneda[k]$};
      \node(01) at (2.5,1) {$C$};
      \node(10) at (0,-1) {$\Yoneda[k]$};
      \node(11) at (2.5,-1) {$D$};
      
      \draw [->] (00) to node[auto, labelsize] {${u}$} (01); 
      \draw [->] (01) to node[auto, labelsize] {$r$} (11); 
      \draw [->] (00) to node[auto,swap,labelsize] {$m_k$} (10); 
      \draw [->] (10) to node[auto,swap,labelsize] {$v$} (11); 
      
      \draw [->] (10) to node[midway, fill=white,labelsize] {$\kappa(k,{u},{v})$} (01); 
\end{tikzpicture}
\qquad\qquad
\begin{tikzpicture}[baseline=-\the\dimexpr\fontdimen22\textfont2\relax ]
      \node(00) at (0,1) {$\partial \Yoneda[{k+1}]$};
      \node(01) at (2.5,1) {$\nameof{C}$};
      \node(10) at (0,-1) {$\Yoneda[{k+1}]$};
      \node(11) at (2.5,-1) {$\nameof{D}$};
      
      \draw [->] (00) to node[auto, labelsize] {$\widehat{u}$} (01); 
      \draw [->] (01) to node[auto, labelsize] {$\nameof{r}$} (11); 
      \draw [->] (00) to node[auto,swap,labelsize] {$m_{k+1}$} (10); 
      \draw [->] (10) to node[auto,swap,labelsize] {$\widehat{v}$} (11); 
      
      \draw [->] (10) to node[midway, fill=white,labelsize] {${\nameof{\kappa}(k+1,\widehat{u},\widehat{v})}$} (01); 
\end{tikzpicture}
\end{equation}
\end{proposition}

This correspondence respects contraction preserving morphisms in the evident sense. In other words, we obtain a functor 
\begin{equation}\label{eqn:nameof_between_Contr}
\nameof{-}\colon \Contr_{D}\longrightarrow \Contr_{\nameof{D}}
\end{equation}
mapping each $(r\colon C\longrightarrow D, \kappa)\in \Contr_{D}$ to $(\nameof{r}\colon\nameof{C}\longrightarrow \nameof{D},\nameof{\kappa})\in\Contr_{\nameof{D}}$, where the contractions $\kappa$ and $\nameof{\kappa}$ are related as in Proposition~\ref{prop:nameof_contraction}.
The functor \eqref{eqn:nameof_between_Contr} is fully faithful and its essential image consists of those objects $(r'\colon C'\longrightarrow\nameof{D},\kappa')\in\Contr_{\nameof{D}}$ such that $C'\in \enGph{\omega}$ has precisely one object.
{As a special case of \eqref{eqn:nameof_between_Contr}, we have 
\begin{equation}\label{eqn:nameof_between_Contr_Tomega}
\nameof{-}\colon \Contr_{T^{(\omega)}1}\longrightarrow \Contr_{(\enGph{T^{(\omega)}})1},
\end{equation}
since $\nameof{T^{(\omega)}1}\cong (\enGph{T^{(\omega)}})1$.}

Now define the category $\OC{\omega}{\enGph{T^{(\omega)}}}$ of \defemph{$(\enGph{T^{(\omega)}})$-operads with contractions} as the pullback
\[
\begin{tikzpicture}[baseline=-\the\dimexpr\fontdimen22\textfont2\relax ]
      \node(00) at (0,1) {$\OC{\omega}{\enGph{T^{(\omega)}}}$};
      \node(01) at (6,1) {$\Mon{\enGph{\omega}/(\enGph{T^{(\omega)}})1}$};
      \node(10) at (0,-1) {$\Contr_{(\enGph{T^{(\omega)}})1}$};
      \node(11) at (6,-1) {$\enGph{\omega}/(\enGph{T^{(\omega)}})1$};
      
      \draw [->] (00) to node[auto, labelsize] {} (01); 
      \draw [->] (01) to node[auto, labelsize] {} (11); 
      \draw [->] (00) to node[auto,swap,labelsize] {} (10); 
      \draw [->] (10) to node[auto,swap,labelsize] {} (11); 

      \draw (0.3,0.2) to (0.8,0.2) to (0.8,0.7);
\end{tikzpicture}
\]
(cf.~\eqref{eqn:OC_as_pullback}). The three fully faithful functors \eqref{eqn:nameof_between_collection}, \eqref{eqn:nameof_between_operad} and \eqref{eqn:nameof_between_Contr_Tomega}
induce a functor
\begin{equation*}
\nameof{-}\colon \OC{\omega}{T^{(\omega)}}\longrightarrow \OC{\omega}{\enGph{T^{(\omega)}}}.
\end{equation*}
This functor is again fully faithful, and its essential image consists of those $(\enGph{T^{(\omega)}})$-operads with contractions whose underlying $\omega$-graph has precisely one object.
Mapping $L\in\OC{}{T^{(\omega)}}$ by this, we obtain $\nameof{L}\in\OC{}{\enGph{T^{(\omega)}}}$. By Proposition~\ref{prop:Gph_and_nameof}, we have an isomorphism of monads on $\enGph{\omega}$
\begin{equation}
\label{eqn:iso}
	\enGph{(L\ast_{T^{(\omega)}}(-))}\cong \nameof{L}\ast_{\enGph{T^{(\omega)}}}(-).
\end{equation}

\medskip

Next observe that for each $n\in\N$ there exists a canonical monad morphism {$\phi^{(n)}\colon\enGph{T^{(n)}}\longrightarrow T^{(n+1)}$ defined as the composition
\[
\begin{tikzpicture}[baseline=-\the\dimexpr\fontdimen22\textfont2\relax ]
      \node(0) at (0,0.4) {};
      \node(1) at (0,-4.4) {};
      \node(b) at (1.2,0.5) {$\enGph{(\enGph{n})}=\enGph{(n+1)}$};
      \node(c) at (0,-1) {$\enGph{(\enCat{n})}$};
      \node(d) at (2,-2) {$\enCat{(\enCat{n})}$};
      \node(e) at (0,-3) {$\enGph{(\enCat{n})}$};
      \node(f) at (1.2,-4.5) {$\enGph{(\enGph{n})}=\enGph{(n+1)}$};
      
      \draw [->] (0) to node[auto, labelsize] {$\enGph{F^{(n)}}$} (c); 
      \draw [->,bend left=30] (c.east) to node[auto, labelsize] {$F^{(\enCat{n})}$} (d); 
      \draw [->,bend left=30] (d) to node[auto, labelsize] {$U^{(\enCat{n})}$} (e.east); 
      \draw [->,bend right=30] (c) to node[auto,swap, labelsize] {$\id$} (e); 
      \draw [->] (e) to node[auto,labelsize] {$\enGph{U^{(n)}}$} (1); 

      \draw [->,bend right=30] (b.west) to node[auto,swap, labelsize] {$\enGph{T^{(n)}}$} (f.west); 
      \draw [->,bend left=30] (b.east) to node[auto, labelsize] {$T^{(n+1)}$} (f.east); 

      \node at (-1.3,-2) {$=$};
      \draw [2cell] (-0.2,-2) to node[auto,labelsize] {$\eta^{(\enCat{n})}$} (1.1,-2); 
      \node at (3.5,-2) {$=$};
\end{tikzpicture}
\] 
}(cf.~Definition~\ref{def:adj_Fn_Un}), where $\eta^{(\enCat{n})}$ is the unit of the adjunction $F^{(\enCat{n})}\dashv U^{(\enCat{n})}$.
Since $\eta^{(\enCat{n})}$ is a cartesian natural transformation (\cite[Proposition 3.5]{CFP_18}) and the right adjoint functor $\enGph{U^{(n)}}$ preserves pullbacks, $\phi^{(n)}$ is a cartesian natural transformation as well. 
Taking the limit, we obtain a cartesian monad morphism $\phi^{(\omega)}\colon \enGph{T^{(\omega)}}\longrightarrow T^{(\omega)}$.

The following is a standard fact for slice categories of a monoidal category (cf. Proposition~\ref{prop:slice_of_mon_cat}). 

\begin{proposition}
\label{prop:slice_of_mon_cat_base_change}
	Let $\cat{M}$ be a monoidal category with pullbacks,  $T$ and $S$ be monoids in $\cat{M}$, and $h\colon T\longrightarrow S$ be a monoid morphism.
Then there exists a monoidal adjunction 
		\begin{equation*}
		\begin{tikzpicture}[baseline=-\the\dimexpr\fontdimen22\textfont2\relax ]
      \node(0) at (0,0) {$\cat{M}/T$};
      \node(1) at (3,0) {$\cat{M}/S$,};
      
      \draw [->, transform canvas={yshift=5}] (0) to node[auto, labelsize] {$\cat{M}/h$} (1); 
      \draw [<-, transform canvas={yshift=-5}] (0) to node[auto,swap,labelsize] {$h^\ast$} (1); 
      \node [rotate=90] at (1.5,0) {$\vdash$};
\end{tikzpicture}
		\end{equation*}
		where $\cat{M}/h$ maps $(p\colon P\longrightarrow T)$ to $h\circ p$ and $h^\ast$ maps $(q\colon Q\longrightarrow S)$ to the pullback of $q$ along $h$.
\end{proposition}

Since $\phi^{(\omega)}$ is a monoid morphism from $\enGph{T^{(\omega)}}$ to $T^{(\omega)}$ in the monoidal category $[\enGph{\omega},\enGph{\omega}]_\cart$, by Proposition \ref{prop:slice_of_mon_cat_base_change} we obtain a monoidal adjunction 
\[
		\begin{tikzpicture}[baseline=-\the\dimexpr\fontdimen22\textfont2\relax ]
      \node(0) at (0,0) {$[\enGph{\omega},\enGph{\omega}]_\cart/\enGph{T^{(\omega)}}$};
      \node(1) at (7,0) {$[\enGph{\omega},\enGph{\omega}]_\cart/T^{(\omega)}$.};
      
      \draw [->, transform canvas={yshift=5}] (0) to node[auto, labelsize] {$[\enGph{\omega},\enGph{\omega}]_\cart/\phi^{(\omega)}$} (1); 
      \draw [<-, transform canvas={yshift=-5}] (0) to node[auto,swap,labelsize] {$(\phi^{(\omega)})^\ast$} (1); 
	\path (0) to node (m) {} (1);
      \node [rotate=90] at (m) {$\vdash$};
\end{tikzpicture}
\]
Modulo the monoidal equivalences {\eqref{eqn:ev_1} induced by the evaluation at $1$,} this monoidal adjunction is
\begin{equation}
\label{eqn:phi_n_between_slices}
		\begin{tikzpicture}[baseline=-\the\dimexpr\fontdimen22\textfont2\relax ]
      \node(0) at (0,0) {$\enGph{\omega}/(\enGph{T^{(\omega)}})1$};
      \node(1) at (5,0) {$\enGph{\omega}/T^{(\omega)}1$,};
      
      \draw [->, transform canvas={yshift=5}] (0) to node[auto, labelsize] {$\phileft$} (1); 
      \draw [<-, transform canvas={yshift=-5}] (0) to node[auto,swap,labelsize] {$(\phi^{(\omega)}_1)^\ast$} (1); 
      \path (0) to node (m) {} (1);
      \node [rotate=90] at (m) {$\vdash$};
\end{tikzpicture}
\end{equation}
given by postcomposition and pullback of $\phi^{(\omega)}_1\colon (\enGph{T^{(\omega)}})1\longrightarrow T^{(\omega)}1$.

This induces in particular the functor
\begin{equation}
	\label{eqn:phi_n_between_operads}
	(\phi^{(\omega)}_1)^\ast\colon\Mon{\enGph{\omega}/T^{(\omega)}1}\longrightarrow\Mon{\enGph{\omega}/(\enGph{T^{(\omega)}})1}.
\end{equation}

On the other hand, by the universality of pullback we also obtain the functor 
\begin{equation}
	\label{eqn:phi_n_between_Contr}
	(\phi^{(\omega)}_1)^\ast\colon \Contr_{T^{(\omega)}1}\longrightarrow \Contr_{(\enGph{T^{(\omega)}})1};
\end{equation}
{indeed, $(\phi^{(\omega)}_1)^\ast$ is the reindexing (or change-of-base) functor between the fibres of  the fibration 
\[
\mathrm{cod}\colon \Contr(\enGph{\omega},\cat{J}^{(\omega)})\longrightarrow \enGph{\omega}
\] 
(see Remark~\ref{rmk:Contr_G_fibration}) induced by $\phi^{(\omega)}_1$.}

The functors $(\phi^{(\omega)}_1)^\ast$ in \eqref{eqn:phi_n_between_slices}, \eqref{eqn:phi_n_between_operads} and \eqref{eqn:phi_n_between_Contr} induce the functor 
\begin{equation*}
	(\phi^{(\omega)}_1)^\ast\colon\OC{n+1}{T^{(\omega)}}\longrightarrow \OC{n+1}{\enGph{T^{(\omega)}}}.
\end{equation*}

\medskip

Now, the $T^{(\omega)}$-operad  $L$ is mapped by $(\phi^{(\omega)}_1)^\ast$ to the $(\enGph{T^{(\omega)}})$-operad $(\phi^{(\omega)}_1)^\ast L$, and by construction we have the following pullback square in the category $\Mon{[\enGph{\omega},\enGph{\omega}]_\cart}$ of cartesian monads on $\enGph{\omega}$:
\begin{equation}
\label{eqn:comp1}
\begin{tikzpicture}[baseline=-\the\dimexpr\fontdimen22\textfont2\relax ]
      \node(01) at (2,1) {$((\phi^{(\omega)}_1)^\ast L)\ast_{\enGph{T^{(\omega)}}} (-)$};
      \node(11) at (2,-1) {${\enGph{T^{(\omega)}}}$};
      \node(20) at (6,1) {$L\ast_{T^{(\omega)}} (-)$};
      \node(21) at (6,-1) {${T^{(\omega)}}.$};
      
      \draw [->] (01) to node[auto, labelsize] {} (11); 
      \draw [->] (01) to node[auto, labelsize] {} (20);
      \draw [->] (11) to node[auto,swap, labelsize] {$\phi^{(\omega)}$} (21);
      \draw [->] (20) to node[auto, labelsize] {$\overline{\ar{L}}$} (21);
            \draw (2.3,0.2) to (2.8,0.2) to (2.8,0.7);
\end{tikzpicture}
\end{equation}
On the other hand, $(\phi^{(\omega)}_1)^\ast L\in\OC{}{\enGph{T^{(\omega)}}}$ is in the essential image of $\nameof{-}$, and hence by the initiality of $L$ we obtain a canonical $(\enGph{T^{(\omega)}})$-operad morphism $\nameof{L}\longrightarrow(\phi^{(\omega)}_1)^\ast L$, giving rise to a monad morphism 
\begin{equation}
\label{eqn:comp2}
		\nameof{L}\ast_{\enGph{T^{(\omega)}}}(-)\longrightarrow ((\phi_1^{(\omega)})^\ast L)\ast_{\enGph{T^{(\omega)}}}(-).
\end{equation}
Precomposing \eqref{eqn:iso} and postcomposing the top horizontal arrow in \eqref{eqn:comp1} with this, we obtain a monad morphism 
\[
\enGph{(L\ast_{T^{(\omega)}}(-))}\longrightarrow L\ast_{T^{(\omega)}}(-),
\]
thus inducing the forgetful functor $\Us$ as desired.

On the level of operads, we have the following diagram in $\enGph{\omega}$, in which the top horizontal composite captures the essence of $\Us$:
\begin{equation}
\label{eqn:outline_operad}
\begin{tikzpicture}[baseline=-\the\dimexpr\fontdimen22\textfont2\relax ]
      \node(00) at (-1,1) {$\nameof{L}$};
      \node(01) at (2,1) {$(\phi^{(\omega)}_1)^\ast L$};
      \node(11) at (2,-1) {${(\enGph{T^{(\omega)}})1}$};
      \node(20) at (5,1) {$L$};
      \node(21) at (5,-1) {${T^{(\omega)}1}$.};

      \draw [->] (00) to node[auto, labelsize] {} (01); 
      \draw [->] (01) to node[auto, swap,labelsize] {} (11); 
      \draw [->] (00) to node[auto,swap,labelsize] {$\ar{\nameof{L}}$} (11);  
      \draw [->] (01) to node[auto, labelsize] {} (20);
      \draw [->] (11) to node[auto,swap, labelsize] {$\phi^{(\omega)}_1$} (21);
      \draw [->] (20) to node[auto, labelsize] {$\ar{L}$} (21);
            \draw (2.3,0.2) to (2.8,0.2) to (2.8,0.7);
\end{tikzpicture}
\end{equation}

Note that by the functoriality of $\Us\colon \WkCats{\omega}\longrightarrow \enGph{(\WkCats{\omega})}$, we see that any strict $\omega$-functor $F\colon \cat{A}\longrightarrow\cat{B}$ induces a family of strict $\omega$-functors $(F_{x,y}\colon\cat{A}(x,y)\longrightarrow\cat{B}(Fx,Fy))_{x,y\in\ob{\cat{A}}}$ as its action on homs.

\begin{remark}
By essentially the same argument, we can also construct the finite-dimensional versions of the forgetful functor $\Us$, namely for each $n\in\N$ a functor
\[
U_{\mathrm{s}}^{(n)}\colon \WkCats{(n+1)}\longrightarrow \enGph{(\WkCats{n})};
\] 
see \cite[Section 9.3]{Leinster_book} or \cite{CFP_18} for the relevant definiton of weak $n$-category.

Also, since $\Us$ (resp.~$U_{\mathrm{s}}^{(n)}$ for each $n\in\N$) is induced from a monad morphism between finitary monads on a locally finitely presentable category $\enGph{\omega}$ (resp.~$\enGph{(n+1)}$), it is monadic. 
In particular, for finite-dimensional versions this means that one can \emph{in principle} define a weak $(n+1)$-category in the sense of Leinster by means of a set of objects, for each pair of objects, a hom weak $n$-category, and various (horizontal) composition operations, following the same (weakened enrichment) approach as the classical definitions of bicategory \cite{Benabou_bicat} and tricategory \cite{GPS}. 
However, an explicit description of the monad induced by $U_{\mathrm{s}}^{(n)}$ seems challenging.
\end{remark}

\begin{remark}
The ($\enGph{T^{(\omega)}}$)-operad with contraction $\nameof{L}$ is in fact the initial object in $\OC{}{\enGph{T^{(\omega)}}}$. This is because for any ($\enGph{T^{(\omega)}}$)-operad with contraction $O=((O,\ar{O}),e,m,\kappa)$, there is a unique morphism $\nameof{L}\longrightarrow O$ whose action on the unique object $\ast$ of $\nameof{L}$ is determined by $e$, and whose action on the hom $\nameof{L}(\ast,\ast)=L$ is determined by the initiality of $L$ in $\OC{}{T^{(\omega)}}$.
This gives an alternative view to the morphism \eqref{eqn:comp2}. 
\end{remark}

{
\begin{remark}
Just as a monoidal category can be seen as a one-object bicategory, a possible definition of \emph{monoidal weak $\omega$-category} would simply be a one-object weak $\omega$-category; cf.~\cite{Baez-Dolan-categorification} and \cite[Chapter~5]{Benjamin_thesis}. If we adopt this definition, then our construction of hom weak $\omega$-categories specialises to the expected operation of \emph{forgetting the monoidal structure}, i.e., taking the underlying weak $\omega$-category of a monoidal weak $\omega$-category.
\end{remark}
}

\section{Garner's definition of weak \texorpdfstring{$\omega$}{omega}-functor}
\label{sec:Garner_weak_functor}
The morphisms of the (Eilenberg--Moore) category $\WkCats{\omega}$ preserve the structures of weak $\omega$-categories on the nose, hence they are called strict $\omega$-functors. Garner \cite{Garner_homomorphisms} introduced the more general notion of \emph{homomorphism} between weak $\omega$-categories, which we call \emph{weak $\omega$-functor}. Weak $\omega$-functors are higher dimensional analogues of pseudofunctors between bicategories or trihomomorphisms between tricategories, i.e., functors preserving the structures up to coherent weakly invertible cells.
Our construction of hom weak $\omega$-categories is compatible not only with strict $\omega$-functors, but also with weak $\omega$-functors.
That is, any weak $\omega$-functor $F\colon\cat{A}\longrightarrow\cat{B}$ induces a family $(F_{x,y}\colon \cat{A}(x,y)\longrightarrow\cat{B}(Fx,Fy))_{x,y\in\ob{\cat{A}}}$ of weak $\omega$-functors between the hom weak $\omega$-categories; we shall show this in the next section. In this section, we review the definition of weak $\omega$-functor.

\medskip

Weak $\omega$-categories and weak $\omega$-functors form a category, which we denote by $\WkCat{\omega}$. 
The category $\WkCat{\omega}$ is defined as the (co-)Kleisli category of a certain comonad $Q$ on $\WkCats{\omega}$. So our main task is to define this comonad. 
In \cite{Garner_homomorphisms}, Garner derives $Q$ from his theory of \emph{algebraic weak factorisation systems} \cite{Garner_understanding}; although this general perspective is intriguing, it presupposes rather heavy machinery. 
Here we shall present a more direct definition.

A certain class of strict $\omega$-functors, which we call \emph{surjective equivalences}, is a key for the definition of weak $\omega$-functor. Surjective equivalences are called \emph{acyclic fibrations} by Garner \cite{Garner_homomorphisms}, and indeed, the surjective equivalences between strict $\omega$-categories are the acyclic (trivial) fibrations {with respect to the \emph{folk model structure} on the category of} strict $\omega$-categories defined in \cite{Lafont_Metayer_Worytkiewicz_folk_model_str_omega_cat} {and further studied in \cite{Ara_Lucas_monoidal_model_str}}.
In the following definitions, we denote the forgetful functor $\WkCats{\omega}\longrightarrow\enGph{\omega}$ by $\mathopen{|}-\mathclose{|}$. Also recall the set $\cat{J}^{(\omega)}$ of morphisms in $\enGph{\omega}$ and related notions introduced in Section \ref{subsec:contractions}.
\begin{definition}
    A strict $\omega$-functor $F\colon\cat{A}\longrightarrow\cat{B}$ is a \defemph{surjective equivalence} if the morphism $|F|\colon |\cat{A}|\longrightarrow |\cat{B}|$ of $\omega$-graphs has the right lifting property with respect to all morphisms in the set $\mathcal{J}^{(\omega)}$.
\end{definition}

\begin{definition}
	Let $F\colon \cat{A}\longrightarrow \cat{B}$ be a strict $\omega$-functor.
	A \defemph{contraction on $F$} is a contraction (with respect to $\mathcal{J}^{(\omega)}$) on the morphism $|F|\colon |\cat{A}|\longrightarrow|\cat{B}|$ of $\omega$-graphs.

For each weak $\omega$-category $\cat{A}$, define the category $\Contr_\cat{A}$ of all strict $\omega$-functors to $\cat{A}$ equipped with contractions as the following pullback of categories:
\begin{equation*}
\begin{tikzpicture}[baseline=-\the\dimexpr\fontdimen22\textfont2\relax ]
      \node(00) at (0,1) {$\Contr_{\cat{A}}$};
      \node(01) at (4,1) {${\WkCats{\omega}/\cat{A}}$};
      \node(10) at (0,-1) {$\Contr_{|\cat{A}|}$};
      \node(11) at (4,-1) {$\enGph{\omega}/|\cat{A}|$};
      
      \draw [->] (00) to node[auto, labelsize] {} (01); 
      \draw [->] (01) to node[auto, labelsize] {$\mathopen{|}-\mathclose{|}$} (11); 
      \draw [->] (00) to node[auto,swap,labelsize] {} (10); 
      \draw [->] (10) to node[auto,swap,labelsize] {} (11); 

      \draw (0.3,0.2) to (0.8,0.2) to (0.8,0.7);
\end{tikzpicture}
\end{equation*}
where $\Contr_{|\cat{A}|}=\Contr (\enGph{\omega},\cat{J}^{(\omega)})_{|\cat{A}|}$.
\end{definition}

Of course, for every object $(F\colon \cat{B}\longrightarrow \cat{A},\kappa)\in\Contr_{\cat{A}}$, $F$ is a surjective equivalence.

\begin{remark}
Using the left adjoint $F^{(\omega)}_\mathrm{Wk}\colon \enGph{\omega}\longrightarrow\WkCats{\omega}$ of $\mathopen{|}-\mathclose{|}$, we obtain the set 
\[
F^{(\omega)}_\mathrm{Wk}\cat{J}^{(\omega)}=\{\, F^{(\omega)}_\mathrm{Wk}m_k\colon F^{(\omega)}_\mathrm{Wk}\partial \Yoneda [k]\longrightarrow F^{(\omega)}_\mathrm{Wk}\Yoneda [k]\mid k\in\mathbb{N}\,\}
\]
of morphisms in $\WkCats{\omega}$. A contraction on a strict $\omega$-functor defined above corresponds to a contraction with respect to $F^{(\omega)}_\mathrm{Wk}\cat{J}^{(\omega)}$, as defined in Definition \ref{def:contraction_general}.
Also, for each weak $\omega$-category $\cat{A}$, the category $\Contr_{\cat{A}}$ is isomorphic to $\Contr (\WkCats{\omega},F^{(\omega)}_\mathrm{Wk}\cat{J}^{(\omega)})_\cat{A}$.
\end{remark}

\begin{proposition}
	For any weak $\omega$-category $\cat{A}$, the category $\Contr_{\cat{A}}$ has an initial object $(\varepsilon_\cat{A}\colon Q\cat{A}\longrightarrow \cat{A},\kappa_\cat{A})$.
\end{proposition}
\begin{proof}
This is a special case of \cite[Proposition 2.6]{Garner_homomorphisms}.
\end{proof}

See \cite[Section 5]{Garner_homomorphisms} for a more explicit description of $Q\cat{A}$ by means of computads. Intuitively, $Q\cat{A}$ is obtained from $\cat{A}$ by inductively replacing equalities between various composites of cells by weakly invertible higher-dimensional cells, so that a strict $\omega$-functor $Q\cat{A}\longrightarrow\cat{B}$ amounts to a weak $\omega$-functor $\cat{A}\longrightarrow \cat{B}$; indeed, this is how a weak $\omega$-functor is defined.
We also remark that the construction $\cat{A}\longmapsto Q\cat{A}$ is an extension to weak $\omega$-categories of the \emph{standard resolutions} of strict $\omega$-categories introduced in \cite[Section 4]{Metayer_resolutions_by_polygraphs}.
{The standard resolution $Q\cat{A}$ of a strict $\omega$-category $\cat{A}$ is a cofibrant replacement of $\cat{A}$ with respect to the folk model structure; see \cite{Lafont_Metayer_Worytkiewicz_folk_model_str_omega_cat}. }

We claim that $Q$ extends to a comonad on $\WkCats{\omega}$.
\begin{itemize}
	\item To describe the action of $Q$ on morphisms, suppose $F\colon \cat{A}\longrightarrow \cat{B}$ is a strict $\omega$-functor.
	We have a functor $F^\ast\colon \Contr_{\cat{B}}\longrightarrow \Contr_{\cat{A}}$ defined by pulling back along $F$; see Remark \ref{rmk:Contr_G_fibration}. 
	In particular, $(\varepsilon_\cat{B}\colon Q\cat{B}\longrightarrow \cat{B},\kappa_\cat{B})$ is mapped to the pullback $F^\ast \varepsilon_\cat{B}\colon F^\ast Q\cat{B}\longrightarrow\cat{A}$ of $\varepsilon_\cat{B}$ along $F$, equipped with the contraction induced from $\kappa_\cat{B}$ by the universality of pullback.
	By the initiality of $(\varepsilon_\cat{A},\kappa_\cat{A})$, we obtain a canonical strict $\omega$-functor $Q\cat{A}\longrightarrow F^\ast Q\cat{B}$ making the left triangle in \eqref{eqn:functoriality_of_Q} commute.
	We define $QF\colon Q\cat{A}\longrightarrow Q\cat{B}$ as the top horizontal composite in \eqref{eqn:functoriality_of_Q}.
	\begin{equation}
	\label{eqn:functoriality_of_Q}
	\begin{tikzpicture}[baseline=-\the\dimexpr\fontdimen22\textfont2\relax ]
	\node(0) at (-2.5,1) {$Q\cat{A}$};
      \node(00) at (-0.5,1) {$F^\ast Q\cat{B}$};
      \node(01) at (1.5,1) {$Q\cat{B}$};
      \node(10) at (-0.5,-1) {$\cat{A}$};
      \node(11) at (1.5,-1) {$\cat{B}$};
      
      \draw [->] (0) to node [auto, labelsize] {} (00);
      \draw [->] (0) to node [auto, labelsize, swap] {$\varepsilon_\cat{A}$}(10);
      \draw [->] (00) to node[auto, labelsize] {} (01); 
      \draw [->] (01) to node[auto, labelsize] {$\varepsilon_\cat{B}$} (11); 
      \draw [->] (00) to node[auto,near start,swap,labelsize] {$F^\ast \varepsilon_\cat{B}$} (10); 
      \draw [->] (10) to node[auto,swap,labelsize] {$F$} (11); 
      \draw (-0.2,0.2) to (0.3,0.2) to (0.3,0.7);
\end{tikzpicture}
	\end{equation}
	\item The counit of $Q$ at $\cat{A}$ is given by $\varepsilon_\cat{A}\colon Q\cat{A}\longrightarrow \cat{A}$.
	\item The comultiplication of $Q$ at $\cat{A}$ is given by $\delta_\cat{A}\colon QA\longrightarrow Q^2 A$  as below, induced by the initiality of $(\varepsilon_\cat{A},\kappa_\cat{A})$ (we equip the morphism $\varepsilon_\cat{A}\circ \varepsilon_{Q\cat{A}}$ with the contraction induced from $\kappa_\cat{A}$ and $\kappa_{Q\cat{A}}$; see \cite[Section~2.8]{Bourke_Garner_1}).
	\begin{equation*}
	\begin{tikzpicture}[baseline=-\the\dimexpr\fontdimen22\textfont2\relax ]
	\node(0) at (0,1) {$Q\cat{A}$};
      \node(00) at (2,1) {$Q^2\cat{A}$};;
      \node(10) at (1,-1) {$\cat{A}$};
      
      \draw [->] (0) to node [auto, labelsize] {$\delta_\cat{A}$} (00);
      \draw [->] (0) to node [auto, labelsize, swap] {$\varepsilon_\cat{A}$}(10);
      \draw [->] (00) to node[auto,labelsize] {$\varepsilon_\cat{A}\circ \varepsilon_{Q\cat{A}}$} (10); 
\end{tikzpicture}
	\end{equation*}
\end{itemize}

It is routine to check that the data $(Q,\varepsilon,\delta)$ defines a comonad on $\WkCats{\omega}$.

\begin{definition}
	A \defemph{weak $\omega$-functor} from $\cat{A}$ to $\cat{B}$ is a strict $\omega$-functor $Q\cat{A}\longrightarrow \cat{B}$. The category $\WkCat{\omega}$ of all weak $\omega$-categories and weak $\omega$-functors is defined as the Kleisli category of the comonad $Q$.
\end{definition}

In particular, this means that the identity weak $\omega$-functor on $\cat{A}$ is $\varepsilon_\cat{A}\colon Q\cat{A}\longrightarrow \cat{A}$, and that the composite of $F\colon Q\cat{A}\longrightarrow\cat{B}$ and $G\colon Q\cat{B}\longrightarrow\cat{C}$ is given by 
\[
	\begin{tikzpicture}[baseline=-\the\dimexpr\fontdimen22\textfont2\relax ]
	  \node(0) at (0,0) {$Q\cat{A}$};
     \node(1) at (2,0) {$Q^2\cat{A}$};
      \node(2) at (4,0) {$Q\cat{B}$};
      \node(3) at (6,0) {$\cat{C}.$};
      
      \draw [->] (0) to node [auto, labelsize] {$\delta_\cat{A}$} (1);
      \draw [->] (1) to node [auto, labelsize] {$QF$}(2);
      \draw [->] (2) to node[auto, labelsize] {$G$} (3); 
\end{tikzpicture}
\]

\begin{remark}
By definition, a weak $\omega$-functor $F\colon \cat{A}\longrightarrow\cat{B}$ gives rise to a span in $\WkCats{\omega}$
\begin{equation*}
\begin{tikzpicture}[baseline=-\the\dimexpr\fontdimen22\textfont2\relax ]
	 \node(0) at (0,-0.75) {$\cat{A}$};
     \node(1) at (2,0.75) {$Q\cat{A}$};
     \node(2) at (4,-0.75) {$\cat{B}$};
      
     \draw [<-] (0) to node [auto, labelsize] {$\varepsilon_\cat{A}$} (1);
     \draw [->] (1) to node [auto, labelsize] {$F$}(2);
\end{tikzpicture}
\end{equation*}
whose left leg $\varepsilon_\cat{A}$ is a surjective equivalence equipped with the universal contraction $\kappa_\cat{A}$ (see also \cite{Bourke_Garner_2}). Conversely, every span in $\WkCats{\omega}$
\begin{equation}
\label{eqn:anafunctor_span}
\begin{tikzpicture}[baseline=-\the\dimexpr\fontdimen22\textfont2\relax ]
	 \node(0) at (0,-0.75) {$\cat{A}$};
     \node(1) at (2,0.75)  {$\cat{E}$};
     \node(2) at (4,-0.75) {$\cat{B}$};
      
     \draw [<-] (0) to node [auto, labelsize] {$J$} (1);
     \draw [->] (1) to node [auto, labelsize] {$G$}(2);
\end{tikzpicture}
\end{equation}
whose left leg $J$ is a surjective equivalence, together with a choice of a contraction $\kappa$ on $J$, gives rise to a weak $\omega$-functor from $\cat{A}$ to $\cat{B}$ (because we obtain a unique morphism $H\colon (\varepsilon_\cat{A}\colon Q\cat{A}\longrightarrow \cat{A},\kappa_\cat{A})\longrightarrow (J\colon\cat{E}\longrightarrow\cat{A},\kappa)$ in $\Contr_\cat{A}$ by initiality and hence a weak $\omega$-functor $G\circ H\colon\cat{A}\longrightarrow\cat{B}$).

In ordinary category theory, surjective equivalences amount to surjective-on-objects equivalences of categories, and the spans \eqref{eqn:anafunctor_span} in $\Catone$ with $J$ a surjective equivalence (without a choice of a contraction) correspond to the \emph{anafunctors} of Makkai \cite{Makkai_anafunctor}.
\end{remark}

{
\begin{remark}
Now that we have the hom weak $\omega$-categories $\cat{A}(x,y)$ of a weak $\omega$-category $\cat{A}$, in light of the weakened enrichment approach to weak higher-dimensional categories, it seems natural to seek a way to extract the \emph{composition} weak $\omega$-functors
\[
\circ\colon \cat{A}(y,z)\times\cat{A}(x,y)\longrightarrow \cat{A}(x,z)
\]
from the weak $\omega$-category structure of $\cat{A}$.
Currently we do not know how to define such weak $\omega$-functors, and leave this important construction as future work.
\end{remark}
}

\section{The forgetful functor \texorpdfstring{$U\colon \WkCat{\omega}\longrightarrow \enGph{(\WkCat{\omega})}$}{U}}
\label{sec:forgetful_weak}
Let us denote the canonical (bijective-on-objects) right adjoint functor associated with the Kleisli category $\WkCat{\omega}$ by $J\colon \WkCats{\omega}\longrightarrow\WkCat{\omega}$. 
In this section we extend the action of the forgetful functor $\Us\colon \WkCats{\omega}\longrightarrow\enGph{(\WkCats{\omega})}$ defined in Section~\ref{sec:forgetful_strict} to weak $\omega$-functors.
That is, we construct an extension $U$ of $\Us$ making the diagram
\begin{equation}
\label{eqn:extension_U}
\begin{tikzpicture}[baseline=-\the\dimexpr\fontdimen22\textfont2\relax ]
      \node(00) at (0,1) {$\WkCats{\omega}$};
      \node(01) at (4.5,1) {$\enGph{(\WkCats{\omega})}$};
      \node(10) at (0,-1) {$\WkCat{\omega}$};
      \node(11) at (4.5,-1) {$\enGph{(\WkCat{\omega})}$};
      
      \draw [->] (00) to node[auto, labelsize] {$\Us$} (01); 
      \draw [->] (01) to node[auto, labelsize] {$\enGph{J}$} (11); 
      \draw [->] (00) to node[auto,swap,labelsize] {$J$} (10); 
      \draw [->] (10) to node[auto,swap,labelsize] {$U$} (11); 
\end{tikzpicture}
\end{equation}
commute.

The key observation is that $\enGph{J}$ exhibits  $\enGph{(\WkCat{\omega})}$ as the Kleisli category of the comonad $\enGph{Q}$ on $\enGph{(\WkCats{\omega})}$.
{This is a consequence of the following facts. First, an adjunction $F\dashv U\colon \cat{C}\to \cat{D}$ in $\CAT$ is isomorphic to the Kleisli adjunction for the comonad $FU$ on the category $\cat{C}$ (hence  in particular $\cat{D}$ is isomorphic to the  Kleisli category of $FU$), if and only if $U$ is bijective on objects. Next, the 2-functor $\enGph{(-)}\colon\CAT\longrightarrow\CAT$ preserves (adjunctions and) bijective-on-objects functors. Hence $\enGph{(-)}$ preserves the Kleisli categories of comonads (as well as of monads).}

So \eqref{eqn:extension_U} is a map between Kleisli categories, and in order to obtain $U$, it suffices to equip $\Us$ with the structure of a comonad opfunctor \cite{Street_FTM}, i.e., a suitable natural transformation $\alpha$ as in
\[
\begin{tikzpicture}[baseline=-\the\dimexpr\fontdimen22\textfont2\relax ]
      \node(00) at (0,1) {$\WkCats{\omega}$};
      \node(01) at (4.5,1) {$\enGph{(\WkCats{\omega})}$};
      \node(10) at (0,-1) {$\WkCats{\omega}$};
      \node(11) at (4.5,-1) {$\enGph{(\WkCats{\omega})}$};
      
      \draw [->] (00) to node[auto, labelsize] {$\Us$} (01); 
      \draw [->] (01) to node[auto, labelsize] {$\enGph{Q}$} (11); 
      \draw [->] (00) to node[auto,swap,labelsize] {$Q$} (10); 
      \draw [->] (10) to node[auto,swap,labelsize] {$\Us$} (11); 

      \draw [2cell] (3.25,0) to node[auto,swap,labelsize] {$\alpha$} (1.25,0); 
\end{tikzpicture}
\]
respecting the structures of the comonads $Q$ and $\enGph{Q}$.
The natural transformation $\alpha$ consists of, for each weak $\omega$-category $\cat{A}$, a morphism
\[
\alpha_\cat{A}\colon (\enGph{Q})\Us\cat{A}\longrightarrow \Us Q\cat{A}
\]
in $\enGph{(\WkCats{\omega})}$. 
We define $\alpha_\cat{A}$ to be the identity on objects. So it remains to define, for each pair $(x,y)$ of objects of $\cat{A}$, a strict $\omega$-functor
\[
(\alpha_\cat{A})_{x,y}\colon Q(\cat{A}(x,y))\longrightarrow (Q\cat{A})(x,y).
\]

We induce $(\alpha_\cat{A})_{x,y}$ from the initiality used to determine $Q(\cat{A}(x,y))$. 
Recall that, by definition, the weak $\omega$-category $Q(\cat{A}(x,y))$ is equipped with a strict $\omega$-functor $\varepsilon_{\cat{A}(x,y)}\colon Q(\cat{A}(x,y))\longrightarrow\cat{A}(x,y)$ and a contraction $\kappa_{\cat{A}(x,y)}$ thereon, such that $(\varepsilon_{\cat{A}(x,y)},\kappa_{\cat{A}(x,y)})$ is the initial object of $\Contr_{\cat{A}(x,y)}$. 
On the other hand, we also have a strict $\omega$-functor $\varepsilon_\cat{A}\colon Q\cat{A}\longrightarrow\cat{A}$ equipped with a contraction $\kappa_\cat{A}$.
We obtain a strict $\omega$-functor $(\varepsilon_\cat{A})_{x,y}\colon (Q\cat{A})(x,y)\longrightarrow \cat{A}(x,y)$ as a part of its action on homs.
Moreover, it is easy to see that the contraction $\kappa_\cat{A}$ restricts to give a contraction $(\kappa_\cat{A})_{x,y}$ on $(\varepsilon_\cat{A})_{x,y}$.
So we obtain an object $((\varepsilon_\cat{A})_{x,y},(\kappa_\cat{A})_{x,y})$ of $\Contr_{\cat{A}(x,y)}$.
Now define the strict $\omega$-functor $(\alpha_\cat{A})_{x,y}$ as the unique morphism $(\varepsilon_{\cat{A}(x,y)},\kappa_{\cat{A}(x,y)})\longrightarrow ((\varepsilon_\cat{A})_{x,y},(\kappa_\cat{A})_{x,y})$ from the initial object in $\Contr_{\cat{A}(x,y)}$.
One can check that $(\Us,\alpha)$ is a comonad opfunctor by a straightforward calculation.

\section{Restriction to weak \texorpdfstring{$\omega$}{omega}-groupoids}
\label{sec:wk_omega_groupoid}
We have shown that a weak $\omega$-category $\cat{A}$ has weak $\omega$-categories $\cat{A}(x,y)$ as homs. 
In this section we briefly sketch that if $\cat{A}$ is a weak $\omega$-\emph{groupoid}, then so are the homs.

Weak $\omega$-groupoids are weak $\omega$-categories in which each $k$-cell ($k\geq 1$) is weakly invertible, so we start with a definition of weakly invertible cell in a weak $\omega$-category. 
In fact, one can define weakly invertible cells in any $\omega$-graph equipped with suitable identity and binary composition operations.

\begin{definition}[{\cite[Definition 1]{Cheng_dual}}]
	An \defemph{$\omega$-precategory} is an $\omega$-graph $P$ equipped with the following structure:
\begin{enumerate}
	\item for any $k\geq 1$ and a $(k-1)$-cell $a$ of $P$, a specified $k$-cell $\id_{a}\colon a\longrightarrow a$;
	\item for any $k\geq 1$ and a pair of $k$-cells $f\colon a\longrightarrow b$ and $g\colon b\longrightarrow c$ of $P$, a specified $k$-cell $g\circ f\colon a\longrightarrow c$.\qedhere
\end{enumerate}
\end{definition}
Note that an $\omega$-precategory has compositions of $k$-cells only along $(k-1)$-dimensional boundaries.

\begin{definition}[\cite{Cheng_dual,vandenBerg_Garner}]
	Let $P$ be an $\omega$-precategory. The set of \defemph{weakly invertible} cells in $P$ is the set of cells of dimension $\geq 1$ defined \emph{coinductively} as follows: for $k\geq 1$, a $k$-cell $f\colon a\longrightarrow b$ of $P$ is weakly invertible if and only if there exist a $k$-cell $g\colon b\longrightarrow a$ and weakly invertible $(k+1)$-cells $\eta\colon \id_{a}\longrightarrow g\circ f$ and $\varepsilon \colon f\circ g\longrightarrow \id_b$.
\end{definition}
Since this is a coinductive definition, in order to show that a $k$-cell $f$ is weakly invertible, it suffices to exhibit a set $W$ of cells in $P$ (called a set of \emph{witnesses} in \cite[Definition 6]{Cheng_dual}) such that $f\in W$ and, for any $f'\colon a'\longrightarrow b'$ in $W$, there exist cells $g'\colon b'\longrightarrow a',\eta'\colon \id_{a'}\longrightarrow g'\circ f'$ and $\varepsilon'\colon f'\circ g' \longrightarrow \id_{b'}$ in $W$.

Each weak $\omega$-category $\cat{A}$ has a canonical $\omega$-precategory structure underlying it. To see this, it suffices to find $k$-cells $i_k$ and $m_k$ in $L$ ($k\geq 1$) of suitable arities, for then we can use their interpretations on the underlying $\omega$-graph $|\cat{A}|$ to define an $\omega$-precategory structure on $|\cat{A}|$.
So we first introduce $k$-cells $0_k$ and $2_k$ in $T^{(\omega)}1$ ($k\geq 1$) which are the arities of $i_k$ and $m_k$ respectively (cf.~\cite[Section 2.1]{vandenBerg_Garner}). For convenience, we also introduce a $k$-cell $1_k$ in $T^{(\omega)}1$ ($k\geq 0$).
{In low dimensions, these cells represent the following globular pasting schemes (cf.~\eqref{eqn:globular_pasting_shcemes}).}
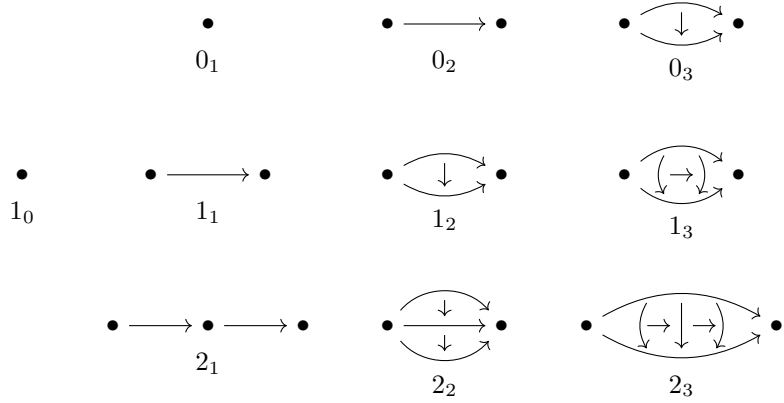
\[
\begin{tikzpicture}[baseline=-\the\dimexpr\fontdimen22\textfont2\relax ]
      \node(21) at (0,0) {$\bullet$};
      \node(l1) at (0,-0.5) {$1_0$};
\end{tikzpicture}
\qquad
\begin{tikzpicture}[baseline=-\the\dimexpr\fontdimen22\textfont2\relax ]
      \node(11) at (0.75,2) {$\bullet$};
      \node(l0) at (0.75,1.5) {$0_1$};
      \node(21) at (0,0) {$\bullet$};
      \node(22) at (1.5,0) {$\bullet$};
      \node(l0) at (0.75,-0.5) {$1_1$};
      \node(31) at (-0.5,-2) {$\bullet$};
      \node(32) at (0.75,-2) {$\bullet$};
      \node(33) at (2,-2) {$\bullet$};
      \node(l0) at (0.75,-2.5) {$2_1$};
      
      \draw [->]  (21) to  (22);      

      \draw [->]  (31) to (32);
      \draw [->] (32) to  (33);
\end{tikzpicture}
\qquad
\begin{tikzpicture}[baseline=-\the\dimexpr\fontdimen22\textfont2\relax ]
      \node(11) at (0,2) {$\bullet$};
      \node(12) at (1.5,2) {$\bullet$};
      \node(l0) at (0.75,1.5) {$0_2$};
      \node(21) at (0,0) {$\bullet$};
      \node(22) at (1.5,0) {$\bullet$};
      \node(l0) at (0.75,-0.6) {$1_2$};
      \node(31) at (0,-2) {$\bullet$};
      \node(32) at (1.5,-2) {$\bullet$};
      \node(l0) at (0.75,-2.8) {$2_2$};
      
      \draw [->]  (11) to (12);
      
      \draw [->,bend left=30]  (21) to node (2u) {} (22);      
      \draw [->,bend right=30] (21) to node (2b) {} (22); 
      \draw [->] (2u) to (2b);

      \draw [->,bend left=50]  (31) to node (3u) {} (32);
      \draw [->] (31) to node (3m) {} (32);
      \draw [->,bend right=50] (31) to node (3b) {} (32); 
      \draw [->] (3u) to (3m);
      \draw [->] (3m) to (3b);
\end{tikzpicture}
\qquad
\begin{tikzpicture}[baseline=-\the\dimexpr\fontdimen22\textfont2\relax ]
      \node(11) at (0,2) {$\bullet$};
      \node(12) at (1.5,2) {$\bullet$};
      \node(l0) at (0.75,1.4) {$0_3$};
      \node(21) at (0,0) {$\bullet$};
      \node(22) at (1.5,0) {$\bullet$};
      \node(l0) at (0.75,-0.7) {$1_3$};
      \node(31) at (-0.5,-2) {$\bullet$};
      \node(32) at (2,-2) {$\bullet$};
      \node(l0) at (0.75,-2.8) {$2_3$};
      
      \draw [->,bend left=30]  (11) to node (1u) {} (12);      
      \draw [->,bend right=30] (11) to node (1b) {} (12); 
      \draw [->] (1u) to (1b);
      
      \draw [->,bend left=40]  (21) to node (2u) {} (22);      
      \draw [->,bend right=40] (21) to node (2b) {} (22); 
      \draw [->,transform canvas={xshift=-0.45em},bend right=30] (2u) to node (2c) {} (2b);
      \draw [->,transform canvas={xshift=0.45em},bend left=30] (2u) to node (2d) {} (2b);
      \draw [->] (0.6,0) to (0.9,0);

      \draw [->,bend left=30]  (31) to node (3u) {} (32);
      \draw [->,bend right=30] (31) to node (3b) {} (32); 
      \draw [->,transform canvas={xshift=-1.1em},bend right=30] (3u) to (3b);
      \draw [->] (3u) to (3b);
      \draw [->,transform canvas={xshift=1.1em},bend left=30] (3u) to (3b);
      \draw [->] (0.3,-2) to (0.6,-2);
      \draw [->] (0.9,-2) to (1.2,-2);
\end{tikzpicture}
\] 
Formally, {in terms of lists (cf.~\eqref{eqn:Tomega-globular-set}), we define recursively $1_0=\bullet$, $1_{k+1}=[1_k]$, $0_1=[\,]$, $0_{k+1}=[0_k]$, $2_1=[\bullet,\bullet]$ and $2_{k+1}=[2_k]$.}
Note that we have $0_k\colon 1_{k-1}\longrightarrow 1_{k-1}$ and $2_k\colon 1_{k-1}\longrightarrow 1_{k-1}$ in $T^{(\omega)}1$.

In order to define the cells $i_k$ and $m_k$ of $L$, we use the contraction $\kappa$ associated with $L$. Note that the unit morphism $e\colon I\longrightarrow (L,\ar{L})$ of the $T^{(\omega)}$-operad $L$ yields a $k$-cell $e_k$ in $L$ of arity $1_k$ for each $k\in\N$. 
We define $i_k=\kappa(k, (e_{k-1},e_{k-1}), 0_k)$ and $m_k=\kappa(k, (e_{k-1},e_{k-1}), 2_k)$. 

We define the \defemph{weakly invertible} cells in a weak $\omega$-category to be the weakly invertible cells in its underlying $\omega$-precategory, and define a \defemph{weak $\omega$-groupoid} to be a weak $\omega$-category in which all cells of dimension $\geq 1$ are weakly invertible \cite{Cheng_dual,vandenBerg_Garner}.

We claim that for any weak $\omega$-groupoid $\cat{A}$ and pair of objects $x$ and $y$, the hom weak $\omega$-category $\cat{A}(x,y)$ given by the forgetful functor $\Us$ of Section \ref{sec:forgetful_strict} (or equivalently, by $U$ of Section \ref{sec:forgetful_weak}) is again a weak $\omega$-groupoid. 
In order to show this, it suffices to show that the canonical identity and binary composition operations in $\cat{A}(x,y)$ agrees with those in $\cat{A}$.
This follows from the construction in Section \ref{sec:forgetful_strict}; since the morphism $\nameof{L}\longrightarrow L$ in \eqref{eqn:outline_operad} preserves contractions and units, for each $k\geq 2$ it maps the $k$-cell $i_{k-1}$ (resp.~$m_{k-1}$) in $\nameof{L}$ to the $k$-cell $i_{k}$ (resp.~$m_k$) in $L$.

If we denote by $\WkGpds{\omega}$ (resp.~$\WkGpd{\omega}$) the full subcategory of $\WkCats{\omega}$ (resp.~$\WkCat{\omega}$) consisting of all weak $\omega$-groupoids,
then it follows that the forgetful functor $\Us$ (resp.~$U$) restricts to
\[
\WkGpds{\omega}\longrightarrow \enGph{(\WkGpds{\omega})}
\qquad
(\text{resp.~}\WkGpd{\omega}\longrightarrow\enGph{(\WkGpd{\omega})}).
\]

\bibliographystyle{plain}
\bibliography{mybib}

\end{document}